\documentclass[a4paper,cleveref,thm-restate]{lipics-v2021}
\nolinenumbers

\graphicspath{{figures/}}

\DeclareMathOperator{\ch}{ch}
\DeclareMathOperator{\cross}{cr}

\hideLIPIcs

\title{Improving the Crossing Lemma by Characterizing Dense 2-Planar and 3-Planar Graphs\footnote{This is a refined version of the paper \cite{LIPIcs:BuengenerKaufmann} that will appear in the proceedings of the 32nd International Symposium on Graph Drawing and Network Visualization, 2024.}}

\author{Aaron B{\"u}ngener}{Universit{\"a}t T{\"u}bingen}{aaron.buengener@student.uni-tuebingen.de}{}{}
\author{Michael Kaufmann}{Universit{\"a}t T{\"u}bingen}{michael.kaufmann@uni-tuebingen.de}{}{}

\authorrunning{A. B\"ungener, M. Kaufmann}
\Copyright{Aaron B\"ungener, Michael Kaufmann}



\begin{document}

\maketitle
\begin{abstract}
The classical Crossing Lemma by Ajtai et al.~and Leighton from 1982 gave an important lower bound of $c \frac{m^3}{n^2}$ for the number of crossings in any drawing of a given graph of $n$ vertices and $m$ edges. The original value was $c= 1/100$, which then has gradually been improved. Here, the bounds for the density of $k$-planar graphs played a central role. Our new insight is that for $k=2,3$ the $k$-planar graphs have substantially fewer edges if specific local configurations that occur in drawings of $k$-planar graphs of maximum density are forbidden. Therefore, we are able to derive better bounds for the crossing number $\cross(G)$ of a given graph $G$.
In particular, we achieve a bound of $\cross(G) \ge \frac{37}{9}m-\frac{155}{9}(n-2)$ for the range of $5n < m \le 6n$, while our second bound $\cross(G) \ge 5m - \frac{203}{9}(n-2)$ is even stronger for larger $m>6n$.

For $m > 6.77n$, we finally apply the standard probabilistic proof from the BOOK and obtain an improved constant of $c>1/27.48$ in the Crossing Lemma. Note that the previous constant was $1/29$.
Although this improvement is not too impressive, we
consider our technique as an important new tool, which might be helpful in various other applications. 
    
\end{abstract}

\section{Introduction}

The classical Crossing Lemma by Ajtai et al.\cite{ajtai1982crossing} and Leighton \cite{leighton1983complexity} has been considerably improved constant-wise from $\frac{1}{100}$ in many subsequent works \cite{DBLP:books/daglib/0019107, DBLP:journals/dcg/PachRTT06, DBLP:journals/combinatorica/PachT97}
and for many variants \cite{schaefer2012graph}, such as bipartite graphs \cite{DBLP:journals/corr/abs-1712-09855}, graphs of bounded girth \cite{DBLP:journals/dcg/PachST00}, multigraphs \cite{DBLP:journals/jgaa/KaufmannPTU21, DBLP:journals/dcg/PachT20}, etc.
Sz{\'{e}}kely \cite{DBLP:journals/cpc/Szekely97}
gave an impressive collection of further applications of the Crossing Lemma in discrete geometry.

The gradual improvement of the above mentioned constant has been mainly done by~using the linear  bounds for the number of edges for planar, 1-planar, 2-planar, etc. graphs. $k$-planar graphs have a drawing where each edge is crossed at most $k$ times.
Density bounds~for $k$-planar $n$-vertex graphs have been subject to intensive research in the past.
While planar graphs have at most $3n-6$ edges, the best known upper bounds for 1-planar, 2-planar and 3-planar graphs are
$4n-8$ \cite{von1983bemerkungen}, $5n-10$ \cite{DBLP:journals/combinatorica/PachT97} and $5.5n-11.5$ \cite{DBLP:conf/compgeom/Bekos0R17} respectively; for the corresponding non-simple versions the bounds might slightly differ \cite{DBLP:conf/compgeom/Bekos0R17}. They
have been directly applied for better bounds for the crossing lemma.
The current best constant of $\frac{1}{29}$ uses even the bound for 4-planar graphs \cite{DBLP:journals/comgeo/Ackerman19}, which is $6n-12$.

We will perform a more refined analysis by considering drawings that are in some sense between $k$-planar and $k+1$-planar drawings for $k=1,2$. In their paper from 2006 \cite{DBLP:journals/dcg/PachRTT06}, Pach, Radoicic, Tardos and Tóth used a similar approach to improve the corresponding constant of the Crossing Lemma. They considered the density of 1-planar drawings with a fixed number of crossing-free triangles, a class of drawings between planar and 1-planar in general.

A similar road has been taken in the paper \cite{DBLP:journals/jct/AckermanT07}
about simple quasi-planar graphs. While the general density bound here is $6.5n$, the authors consider drawings without triangular cells that have no vertex on the boundary. For such a more general class, a bound of $7n$ can be derived. This bound has not been applied for the Crossing Lemma, though.
We will apply such a refined look to 2- and 3-planar drawings:
It turns out that either we can prove much smaller bounds for the edge density  than provided by the upper bounds~of the corresponding $k$-planar classes (which is per se good for the Crossing Lemma) or we can characterize the drawing in a very good way, which simplifies the way of counting the~ crossings.

The idea has been motivated by some results in the literature.
(Non-simple) optimal 2-planar and 3-planar graphs have been characterized \cite{DBLP:conf/compgeom/Bekos0R17}, and there is very limited flexibility for the structure of such graphs. We know that with much less restrictions on the drawings, the limits of the maximum density for some superclasses for 1-planar and 2-planar graphs are still roughly at the same value. Examples for this effect are the min-1-planar and min-2-planar graphs \cite{DBLP:conf/gd/BinucciBBDDHKLMT23} as superclasses of 1-planar and 2-planar graphs, as well as gap-planar graphs as a superclass of 2-planar graphs \cite{DBLP:journals/tcs/BaeBCEE0HKMRT18}.

To use the concept of $k$-planarity for various values of $k$, we planned to specify at which point between $k$- and $k+1$-planarity the density is changing. This turned out to be difficult, and hence we go the other way around and forbid local configurations that have to occur in optimal $k$-planar drawings. That leads to nice insights on the density bounds and surprising results.
Note that all our results hold for non-simple graphs and non-simple drawings.

We remark that the current version is a refinement of our paper \cite{LIPIcs:BuengenerKaufmann}. We were able to strengthen a crucial lemma \cref{pro:m_4andm_3} and further improve the bound for the Crossing Lemma.

\section{Definitions and Notation}

A \emph{drawing} or \emph{topological graph} $D$ is a graph drawn in the plane such that the vertices are pairwise distinct points and the edges are represented as Jordan arcs connecting the corresponding endpoints. We assume simplicity in the sense that edges do not overlap other vertices in the interior. Two edges might cross, but we do not allow that more than two edges cross at a single point. We also assume that two edges have only a finite number of common interior points and no two edges meet tangentially.
Remark that we will consider not necessarily simple drawings, i.e., we will allow non-homotopic multiple edges as well as adjacent crossing edges, while loops are forbidden. Since we mostly assume that the number of crossings will be minimal, there will be no empty lenses, i.e., empty regions having a boundary that is being defined by two edges; c.f. \cref{pro:charging-helper}.

The \emph{crossing number} $\cross(D)$ is defined to be the total number of crossing points in $D$. For an abstract graph $G$, the \emph{crossing number} $\cross(G)$ is the minimum value of $\cross(D)$ over all drawings $D$ with $D$ is a drawing of $G$. A drawing $D$ is \emph{$k$-planar} if no edge is crossed more than $k$ times. A graph $G$ is \emph{$k$-planar} if it has a $k$-planar drawing.

\textbf{Forbidden configurations.}
We now define three forbidden configurations that play a key role:
A \emph{full $k$-planar $p$-gon} $F^k_p$
can be described by a $p$-cycle $C_p$ of planar edges with no other vertices inside, which is then greedily extended by a maximal number of edges to be placed inside that are as short as possible observing this subgraph is still $k$-planar. To finally arrive at $F^k_p$, we delete the planar cycle $C_p$ at the boundary.
In this way, we define
a \emph{full 2-planar pentagon} $F^2_5$ to be the graph $K_5 - C_5$ drawn in the way described above (see \cref{fig:forbidden-configuration-a}).
Similarly, we can define full 2-planar hexagons $F^2_6$ and full 3-planar hexagons $F^3_6$ as specific drawings of subgraphs of $K_6-C_6$. More precisely, a \emph{full 2-planar hexagon} consists of the six short, i.e., 2-hop edges inside a planar $C_6$ (see \cref{fig:forbidden-configuration-b}). A \emph{full 3-planar hexagon} consists of all possible 2-hop and two 3-hop edges inside a planar $C_6$ (see \cref{fig:forbidden-configuration-c}).

Clearly, a configuration $F^k_p$ may be crossed by some other edges. But for full 2-planar~penta\-gons and full 2-planar hexagons, this cannot happen in the case of 2-planar drawings, which motivates to define the planar 5-cycle resp. 6-cycle surrounding them as their \emph{boundary}~(even if not all of its edges may exist in a drawing). This implies that, for 2-planar drawings, full 2-planar pentagons and hexagons are edge-disjoint (while they may have common boundary edges). Similarly, in the case of 3-planarity and full 3-planar hexagons, the cycle surrounding them consists of uncrossed edges if there are~no empty lenses. With this in mind, we analogously define the \emph{boundary} of a full 3-planar hexagon, and observe that these configurations are edge-disjoint for 3-planar drawings.

Using the definitions above, we are able to state our main results in the next section.
\begin{figure}
    \begin{subfigure}[b]{0.27\linewidth}
    \center
        \includegraphics[page=1, width = \linewidth]{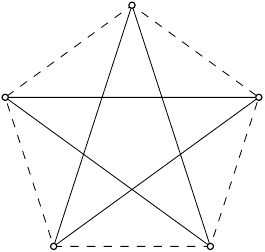}
        \subcaption{}
        \label{fig:forbidden-configuration-a}
    \end{subfigure}  
    \hfill
    \begin{subfigure}[b]{0.27\linewidth}
    \center
        \includegraphics[page=2, width = \linewidth]{definitions.pdf}
        \subcaption{}
        \label{fig:forbidden-configuration-b}
    \end{subfigure} 
    \hfill
    \begin{subfigure}[b]{0.27\linewidth}
    \center
        \includegraphics[page=3, width = \linewidth]{definitions.pdf}
        \subcaption{}
        \label{fig:forbidden-configuration-c}
    \end{subfigure} 
    \hfill
        \caption{(a) A full 2-planar pentagon $F^2_5$, (b) a full 2-planar hexagon $F^2_6$ and (c) a full 3-planar hexagon $F_3^6$ with their boundaries (dashed).
        }
        \label{fig:forbidden-configuration}
\end{figure}

\section{Results}
In this section, we present our results. The proofs of \cref{th:density-2-planar-forbidden-configuration} and \cref{th:density-3-planar-forbidden-configuration} use the discharging method and can be found in \cref{sec:charging}.

\begin{restatable}{theorem}{thdensitytwoplanar}\label{th:density-2-planar-forbidden-configuration} 
Any graph $G$ with $n\ge 3$ vertices that admits a 2-planar $F^2_5$-free drawing has at most $4.5(n-2)$ edges. If the drawing is also $F^2_6$-free, then $G$ has at most $\frac{13}{3}(n-2)$ edges.
\end{restatable}

Counting the number of edges in a drawing consisting of $0.5(n-2)$ full 2-planar hexagons, we see that the first of the two bounds is tight. For the second bound, we refer to a pentagonalization of the plane, where four edges have been added within each pentagon.
 
\begin{corollary}\label{cor:characterization-dense-2-planar}
For every 2-planar drawing of any graph with $n\ge 3$ vertices and $\frac{13}{3}(n-2)+x$ edges for $x \in [0,\frac{2}{3}(n-2)]$, the number of $F^2_5$ and $F^2_6$ configurations is at least $x$.
\end{corollary}

Note that $G$ cannot be 2-planar for $x > \frac{2}{3}(n-2)$ by the corresponding density bound.

\begin{proof}
Assume that drawing $D$ is a 2-planar drawing of a graph with $n$ vertices and $\frac{13}{3}(n-2)+x$ edges such that the number of $F^2_5$ or $F^2_6$ configurations is $y<x$. We can destroy those configurations by removing one edge from each $F_5^2$ and $F_6^2$. Hence, we still have more than $\frac{13}{3}(n-2)$ edges, which is a contradiction to \cref{th:density-2-planar-forbidden-configuration}.
\end{proof}

This implies that drawings of optimal 2-planar graphs consist of $\frac{2}{3}(n-2)$ full 2-planar pentagons, a fact that has been already known \cite{DBLP:conf/compgeom/Bekos0R17}. Similar results hold for 3-planar drawings.

\begin{restatable}{theorem}{thdensitythreeplanar}\label{th:density-3-planar-forbidden-configuration}
Any graph with $n\ge 3$ vertices that admits a 3-planar $F^3_6$-free drawing has at most $5(n-2)$ edges.
\end{restatable}

This bound is tight, which one can see by considering optimal 2-planar graphs.

The next corollary allows us to characterize drawings of dense 3-planar graphs very well. This extends the characterization of optimal 3-planar graphs, which must have a drawing consisting of $\frac{1}{2}(n-2)$ $F^3_6$ configurations and their boundaries \cite{DBLP:conf/compgeom/Bekos0R17}.

\begin{corollary}\label{cor:characterization-dense-3-planar}
For ervery 3-planar drawing of any graph with $n\ge 3$ vertices and $5(n-2)+x$ edges for $x \in [0,0.5(n-2)]$, the number of $F^3_6$ configurations is at least $x$.
\end{corollary}

Note that $G$ cannot be 3-planar for $x > 0.5(n-2)$ by the corresponding density bound.

\begin{proof}
    Analogously to the proof of \cref{cor:characterization-dense-2-planar}, 
    we assume that there is a 3-planar drawing $D$ of a graph with $n$ vertices and $5(n-2)+x$ edges such that the number of $F^2_5$ or $F^2_6$ configurations is $y<x$. Those configurations can be destroyed by removing one edge~from each $F^3_6$, hence we still have more than $5(n-2)$ edges, which is a contradiction to \cref{th:density-3-planar-forbidden-configuration}.
\end{proof}

A consequence of this is a new upper bound for the edge density of simple 3-planar graphs, i.e., the case where multi-edges are forbidden. Note that the best known bound before was $5.5n-11.5$ edges \cite{DBLP:conf/compgeom/Bekos0R17} and there exist examples with $5.5n-15$ edges \cite{DBLP:journals/dcg/PachRTT06}.

\begin{corollary}\label{cor:density-3-planar}
There are no 3-planar graphs on $n\ge 3$ vertices with $5.5n-11.5$ edges. Therefore, any simple 3-planar graph on $n\ge 3$ vertices has at most $5.5n-12$ edges.
\end{corollary}

\begin{proof}
    Assume that there exists a (not necessarily simple) 3-planar graph $G$ with $5.5n-11.5$ edges. Then, by \cref{cor:characterization-dense-3-planar}, we would find in any 3-planar drawing $D$ of $G$ at least $0.5(n-2)-0.5$ full 3-planar hexagons. Let $\mathcal{H}$ be any triangulation on the set of vertices that includes all the boundaries of all $F^3_6$ configurations in $D$. As $F^3_6$ configurations consist of four triangles, only $2(n-2) - 4(0.5(n-2)-0.5) = 2$ triangles in $\mathcal{H}$ do not belong to an $F^3_6$.
    
    Now we count the edges. Starting with the edges of $\mathcal{H}$, each $F^3_6$ consists of five additional edges. The other two triangles may contain one additional edge, which gives in total at most $3(n-2) + 2.5(n-2) - 2.5 + 1 = 5.5n-12.5$ edges, contradicting the assumed density.
\end{proof}

From \cref{th:density-2-planar-forbidden-configuration} and \cref{th:density-3-planar-forbidden-configuration} we can also derive new lower bounds for the number of crossings in a graph. The proof can be found in \cref{sec:linear-bounds}.

\begin{restatable}{theorem}{thlinearbounds}\label{th:linear-bounds}
Let $G$ be a graph with $n>2$ vertices and $m$ edges. Then 
    \begin{enumerate}[(a)]
        \item $\cross(G) \ge \frac{37}{9}m - \frac{155}{9}(n-2)$,
        \item $\cross(G) \ge 5m - \frac{203}{9}(n-2)$.
    \end{enumerate}
\end{restatable}
\medskip

A slightly weaker bound than in (a) of $\cross(G) \ge 4m-\frac{50}{3}(n-2)$ can be derived with a significantly shorter proof by only applying \cref{th:density-3-planar-forbidden-configuration}; we point this out in the proof.

That improves the best known results for $m > 5(n-2)$, which are $\cross(G) \ge 4m - \frac{103}{6}(n-2)$ \cite{DBLP:journals/dcg/PachRTT06} respectively $\cross(G) \ge 5m - \frac{139}{6}(n-2)$ \cite{DBLP:journals/comgeo/Ackerman19}.
\cref{th:linear-bounds} implies directly a better constant in the Crossing Lemma.

\begin{theorem}\label{th:crossing-lemma}
    Let $G$ be a graph with $n$ vertices and $m$ edges. Then $\cross(G) \ge \frac{1500}{41209}\frac{m^3}{n^2} - \frac{54791}{41209}n > \frac{1}{27.48}\frac{m^3}{n^2} - 1.33n$. If $m \ge 6.77n > \frac{203}{30}n$, then $\cross(G) \ge \frac{1500}{41209}\frac{m^3}{n^2} > \frac{1}{27.48}\frac{m^3}{n^2}$.
\end{theorem}

\begin{proof}
Let $G$ be a graph with $n$ vertices and $m$ edges.
For the case $m \ge \frac{203}{30}n$, we 
construct a random subgraph $G'$ by selecting every vertex of $G$ independently with probability $p = \frac{203}{30}n/m\le 1$. We denote the number of edges and vertices in $G'$ by $m'$ and $n'$.
By \cref{th:linear-bounds} and linearity of expectation, we obtain  $\mathbb{E}[\cross(G')] \ge 5\mathbb{E}[m']- \frac{203}{9}\mathbb{E}[n']$. We replace $\mathbb{E}[n'] = pn$, $\mathbb{E}[m'] = p^2m$ and $\mathbb{E}[\cross(G')] = p^4 \cross(G)$, and get 
\[\cross(G) \ge \frac{5m}{p^2} - \frac{203n}{9p^3} = \frac{1500}{41209} \frac{m^3}{n^2}.\]

For the case $m < \frac{203}{30}n$ we compare the bound $\cross(G) \ge \frac{1500}{41209}\frac{m^3}{n^2} - \frac{54791}{41209}n$ with the corresponding best known linear bounds $\cross(G)\ge m -3(n-2)$, $\cross(G) \ge \frac{7}{3}m-\frac{25}{3}(n-2)$ \cite{DBLP:journals/dcg/PachRTT06} and \cref{th:linear-bounds}.
\end{proof}

One direct application of the improved Crossing Lemma is a new bound on the edge density for $k$-planar graphs.

\begin{corollary}\label{cor:density-k-planar}
    For $k\ge 2$, any simple $k$-planar graph with $n$ vertices has at most $3.71\sqrt{k}n$ edges.
\end{corollary}

\begin{proof}
As in \cite{DBLP:journals/combinatorica/PachT97}, the new bound for $k$-planar graphs can be derived directly from the new Crossing Lemma and the fact that each edge can be crossed at most $k$ times:
\[\frac{1}{27.48}\frac{m^3}{n^2} \leq \cross(G) \leq km/2,\]
which then leads to $m \leq \sqrt{13.74 k} n \leq 3.71\sqrt k n$.
\end{proof}

The best previous constant in the bound was $3.81$.

\section{Proof of Theorems \ref{th:density-2-planar-forbidden-configuration} and \ref{th:density-3-planar-forbidden-configuration}}\label{sec:charging}

In this section, we give the proofs of the two central theorems of our paper. First, we will introduce some necessary concepts, we basically adopted the notation by Ackerman \cite{DBLP:journals/comgeo/Ackerman19}.

{\bf Notation:} 
We interpret a drawing $D$ as a plane map $M(D) = (V',E')$ whose vertices $V'$ are either vertices $V(D)$ of $D$ or crossing points of $D$. An edge $e$ in $E'$ connects two vertices of $V'$, i.e., it is a crossing-free segment of an edge of $D$, which we denote by $\overline{e}$. We call an edge of $E'$ an \emph{$r$-edge}, if $r \in \{0,1,2\}$ of its endpoints are vertices of $D$. For a vertex $v \in V(D)$, we write $\deg(v)$ for its \emph{degree}. The degree of a crossing is always four.

Let $F'$ be the set of faces of $M(D)$. For a face $f \in F'$, we write $\vert f \vert$ for the number of edges in $E'$ that are incident to $f$. Similarly, $\vert V(f) \vert $ denotes the number of (real) vertices of $D$ that are incident to $f$. Note that we will assume 2-connectivity, hence the boundary of every face is a simple cycle and we avoid double-counting of the vertices.
A face with $\vert f \vert = s$ is called a \emph{$s$-gon}. In the cases of $s=3,4,5,6,7$ we write instead \emph{triangle, quadrilateral, pentagon, hexagon} and \emph{heptagon}. 
If we want to denote that $\vert V(f)\vert = r$ and $\vert f \vert = s$, we write \emph{$r$-$s$-gon} and use this wording also for $2$-triangles, $0$-quadrilaterals, etc.\ for simplicity. If we only want to specify for a face that $\vert V(f)\vert = r$, then we call it an $r$-face.

Further, we need some definitions for relations between faces in $F'$. Two faces are \emph{$r$-neighbors} if they share an $r$-edge.
Let now be $e_0 \in E'$ a 0-edge of a face $f_0 \in F'$ and $f_1 \in F'$ the 0-neighbor of $f_0$ at $e_0$. For $i \ge 1$, if $f_i \in F'$ is a 0-quadrilateral, then let be $f_{i+1} \in F'$ the 0-neighbor of $f_i$ at the edge $e_i$ opposite to $f_{i-1}$. The face $f_i$, for which $i$ is maximal, is called the \emph{wedge-neighbor} of $f_0$ at $e_0$. Since $D$ is 3-planar, we have $i \le 3$. Notice the alternative definition of a wedge-neighbor by Ackerman \cite{DBLP:journals/comgeo/Ackerman19}.
Finally, we define two faces $f,f' \in F'$ to be \emph{vertex-neighbors}, if $f$ and $f'$ share a crossing-vertex $c$, but not an edge in $E'$ incident to $c$. See \cref{fig:definition-neighbors} for an illustration of the defined terms.

\begin{figure}
    \begin{subfigure}[b]{0.27\linewidth}
    \center
        \includegraphics[page=4, width = \linewidth]{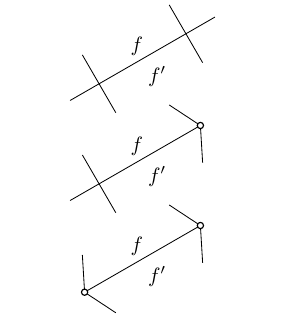}
        \subcaption{}
    \end{subfigure}  
    \hfill
    \begin{subfigure}[b]{0.27\linewidth}
    \center
        \includegraphics[page=5, width = \linewidth]{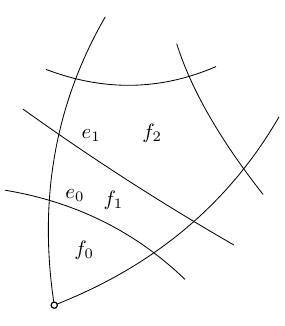}
        \subcaption{}
    \end{subfigure} 
    \hfill
    \begin{subfigure}[b]{0.27\linewidth}
    \center
        \includegraphics[page=6, width = \linewidth]{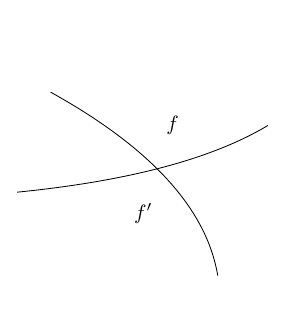}
        \subcaption{}
    \end{subfigure} 
    \hfill
        \caption{Illustrations of the defined neighborhood-relations. (a) From top to bottom: The faces $f$ and $f'$ are 0-neighbors, 1-neighbors, 2-neighbors resp. (b) The 0-pentagon $f_2$ is the wedge-neighbor of the 1-triangle $f_0$ at its edge $e_0$. (c) The faces $f$ and $f'$ are vertex-neighbors.
        }
        \label{fig:definition-neighbors}
\end{figure}

{\bf Preliminaries for the proofs:} 
We prove both theorems by induction. This will allow us, as in \cite{DBLP:journals/comgeo/Ackerman19}, to study only 2-connected drawings (see \cref{pro:2-connected}).
For $n=3$, independently from the forbidden configurations, there are at most three non-homotopic edges in any drawing and therefore both theorems hold. If $n>3$ and there is a vertex $v \in G$ with $\deg(v) \le 4$, then the theorems follow after removing $v$ by induction.

\begin{proposition}\label{pro:2-connected}
    If $D$ is not 2-connected, then \cref{th:density-2-planar-forbidden-configuration} and \cref{th:density-3-planar-forbidden-configuration} are true.
\end{proposition}

\begin{proof}
    The argument follows the lines of \cite{DBLP:journals/comgeo/Ackerman19}. To argue for the different scenarios of \cref{th:density-2-planar-forbidden-configuration} and \cref{th:density-3-planar-forbidden-configuration} at the same time, let $a(n-2)$ for $a \in \{\frac{13}{3},4.5,5\}$ be an upper bound on the number of edges, which we want to prove.
    Assume that there is a vertex $x\in E'$ such that $M(D) \setminus \{x\}$ is not connected. Then $x$ is either a vertex or a crossing of $D$.
    
    If $x$ is a vertex of $D$, then $D \setminus \{x\}$ is not connected, so let $D_1, ..., D_k$ be the connected components of $D \setminus \{x\}$. Let further $D'$ be the drawing induced by $V(D_1) \cup \{x\}$ and $D''$ the drawing induced by $V(D_2) \cup ... \cup V(D_k) \cup \{x\}$. Let $\vert V(D')\vert = n'$, $\vert V(D'')\vert = n''$ and observe $n' + n'' = n+1$. Since every vertex has at least degree four, $4 < n', n'' < n$ holds. By induction, it follows $m \le  (an' - 2a) + (an'' -2a) = a(n+1)-4a < a(n-2)$.
    
    Assume now that $x$ is a crossing of $D$. Let $\hat{D}$ be the drawing obtained by replacing $x$ by a vertex. This increases the number of vertices by one and the number of edges by two. Let $D_1, ..., D_k$ be the connected components of $\hat{D} \setminus \{x\}$. Again, let $D'$ be the drawing induced by $V(D_1) \cup \{x\}$ and $D''$ the drawing induced by $V(D_2) \cup ... \cup V(D_k) \cup \{x\}$. For $\vert V(D')\vert = n'$, $\vert V(D'')\vert = n''$ we observe $4 < n', n'' < n$. By induction, we get $m \le  (an' - 2a) + (an'' -2a) -2 = a(n+2)-4a -2 < a(n-2)$.
\end{proof}

Therefore we will always assume that $D$ is 2-connected. As both theorems consider upper bounds for the number of edges for the specific graph classes, we also assume that we consider graphs $G$ that are edge-maximum for the specific class of graphs, and for such graphs a corresponding drawing $D$ that is crossing-minimum. These assumptions will enable us to conduct a focused analysis of the bounds for the number of edges.

\begin{proposition}\label{pro:charging-helper}
    Let $D$ be a drawing that is either
    (1) 2-planar $F^2_5$-free or
    (2) 2-planar $F^2_5$-free and $F^2_6$-free or
    (3) 3-planar $F^3_6$-free 
    and  maximally-dense-crossing-minimal under this restriction. Then the following properties hold:
    \begin{enumerate}[(a)]
    \item There are no empty lenses.
    \item For all faces $f \in F'$ we have $\vert f \vert \ge 3$.
    \item The wedge-neighbor of a 0-triangle or a 1-triangle is a face $f\in F'$ with $\vert f \vert \ge 4$ that is not a 0-quadrilateral. 
    \item If there are two vertices $u,v \in V(D)$ on the boundary of a face $f \in F'$, then the edge $uv$ is part of the boundary of $f$. Therefore every face $f \in F'$ with $\vert V(f)\vert > 2$ is a 3-triangle.
\end{enumerate}
\end{proposition}
\begin{proof}
    \begin{enumerate}[(a)]     
        \item Since there are no two homotopic edges, there are no empty lenses with two vertices. Any other empty lens can be destroyed by swapping the segments of the edges of $D$ that define the empty lens (without creating one of the forbidden configurations). This reduces the number of crossings contradicting that $D$ is crossing-minimal.
        \item Loops and self-intersecting edges are forbidden, so there is no face $f \in F'$ with $\vert f \vert = 1$. Every face $f \in F'$ with $\vert f \vert = 2$ is an empty lens, which does not appear in $D$ by (a).
        \item Let $f$ be an arbitrary face. By definition, the face $f$ is never a 0-quadrilateral. If $\vert f \vert = 3$, then this would imply an empty lens.
        \item For an arbitrary face $f$, assume that no edge $ e= uv$ exists on the boundary of $f$. Therefore, we may insert $e$ contradicting that $D$ is maximally dense. By this, we cannot create one of the three forbidden configurations $F^2_5, F^2_6$ and $F^3_6$, since they do not contain planar edges.        
        This does not create homotopic edges as every other edge $e'=uv$ homotopic to $e$ would have been already on the boundary of $f$ or would have formed an empty lens with an edge of the boundary of $f$ contradicting (a).

        Assume now that a face $f$ with $\vert V(f) \vert > 2$ exists that is not a 3-triangle. Then we find three vertices in $V(D)$ on the boundary of $f$, which do not all appear next to each other. We introduce a new edge between two of them,  contradicting the maximality of $D$.
    \end{enumerate}
\end{proof}

In the following, we will use the \emph{discharging method}. See \cite{DBLP:journals/comgeo/Ackerman19, DBLP:journals/jct/AckermanT07, DBLP:conf/gd/BinucciBBDDHKLMT23, radoivcic2008discharging} for similar applications of this technique.
We define a \emph{charging function} $\ch: F' \rightarrow \mathbb{R}$ that assigns an \emph{initial charge} of
\begin{equation}\label{eq:initial-charge}
 \ch(f) = \vert f \vert + \vert V(f) \vert - 4
\end{equation}
to every face $f \in F'$. It is known that for the total charge $\sum_{f \in F'} \ch(f) = 4n-8$ holds (refer to~\cite{DBLP:journals/jct/AckermanT07}  for details). The challenge now is to redistribute the charge so that in the end every face $f \in F'$ has a charge of $\ch'(\cdot)$ that satisfies $\ch'(f) \geq \alpha \vert V(f) \vert$ for a suitable $\alpha > 0$, while the total charge does not change. From this and the observation that $\sum_{f \in F'} \vert V(f) \vert = \sum_{v \in V(D)} \deg(v) = 2m$ holds, we can derive an upper bound of
\begin{equation}\label{eq:modified-charge}
    m \leq \frac{2}{\alpha}(n-2)
\end{equation}
on the number of edges. For a given $\alpha$ and a face $f$ with charge $c$, we say that $\vert c- \alpha \vert V(f)\vert \vert$ is the \emph{demand} of $f$, if $c- \alpha \vert V(f)\vert$ is negative, otherwise we call it the \emph{excess} of $f$. If $f$ has no demand, then we also say that $f$ is satisfied.

\subsection{Proof and Discharging for Theorem \ref{th:density-2-planar-forbidden-configuration}}
\thdensitytwoplanar*

\begin{proof} We start with the bound of $\frac{13}{3}(n-2)$. Let $D$ be a 2-planar, $F^2_5$-free and $F^2_6$-free drawing that is maximally-dense-crossing-minimal. Assign to every face $f \in F'$ the initial charge $\ch(f)$ according to \cref{eq:initial-charge}. 
The initial charges are distributed in the following~way:
\begin{itemize}
    \item \textsf{Step 1}: Each 0-triangle receives $\frac{1}{3}$ charge from each of its wedge-neighbors.
    \item \textsf{Step 2}: Each 1-triangle receives $\frac{1}{26}$ charge from both 1-neighbors.
    \item \textsf{Step 3}: Each 1-triangle receives $\frac{5}{13}$ charge from its wedge-neighbor.
    \item \textsf{Step 4}: Each 2-quadrilateral contributes its excess to its wedge-neighbor.
    \item \textsf{Step 5}: 
    For each 2-triangle $f$, let $\mathcal{C}(f)$ be the inclusion-minimal planar cycle of $D$ enclosing $f$ (i.e. the planar cycle that does not contain other planar edges). Then $f$ distributes its excess equally over those faces that lie inside $\mathcal{C}(f)$ and have a demand.
\end{itemize}
Denote the charges after the $i$-th step by $\ch_i(\cdot)$. With this, we have $\ch'(\cdot) = \ch_5(\cdot)$. 

\begin{proposition}\label{pro:2-planar-charging}
    For all faces $f \in F'$, we have $\ch'(f) \ge \frac{6}{13} \vert V(f) \vert$.
\end{proposition}

\begin{proof}
We analyze the final charge $\ch'(\cdot)$ for all faces. Note that a face contributes through each edge of its boundary in \textsf{Step 1-3} at most once and the only contributing faces in \textsf{Step 1} are 2-quadrilaterals (see \cref{fig:density-2-planar-proof-a}) and in \textsf{Step 2} 2-triangles (see \cref{fig:density-2-planar-proof-b}).
\begin{figure}
    \begin{subfigure}[b]{0.24\linewidth}
    \center
    \includegraphics[width=\textwidth, page=1]{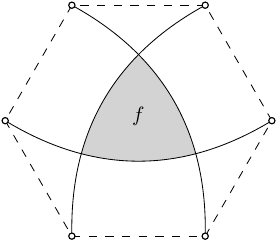}
    \subcaption{}
    \label{fig:density-2-planar-proof-a}
    \end{subfigure}  
    \hfill
    \begin{subfigure}[b]{0.24\linewidth}
    \center
    \includegraphics[width=\textwidth, page=2]{2planarCases.pdf}
    \subcaption{}
    \label{fig:density-2-planar-proof-b}
    \end{subfigure} 
    \hfill
    \begin{subfigure}[b]{0.24\linewidth}
    \center
    \includegraphics[width=\textwidth, page=3]{2planarCases.pdf}
    \subcaption{}
    \label{fig:density-2-planar-proof-c}
    \end{subfigure}
    \hfill
    \begin{subfigure}[b]{0.24\linewidth}
        \center
    \includegraphics[width=\textwidth, page=4]{2planarCases.pdf}
    \subcaption{}
    \label{fig:density-2-planar-proof-d}
    \end{subfigure} 
    \hfill
    \caption{Discharging for \cref{th:density-2-planar-forbidden-configuration}.
    Planar edges that exist by \cref{pro:charging-helper} are dashed.
    }
    \label{fig:density-2-planar-proof}
\end{figure}
Also $\ch_3(f) \ge \frac{6}{13} \vert V(f)\vert$ already implies $\ch'(f) \ge \frac{6}{13} \vert V(f)\vert$. Because of \cref{pro:charging-helper} there are only 3-triangles and faces $f$ with $\vert f \vert \ge 3$ and $\vert V(f)\vert \le 2$.
\begin{itemize}
    \item \emph{$f$ is a 0-triangle.} Then $f$ receives in \textsf{Step 1} in $3 \cdot \frac{1}{3}$ charge and never contributes charge. Therefore $\ch_3(f) = -1 + 1 = 0  \ge \frac{6}{13} \cdot 0$.
    \item \emph{$f$ is a 1-triangle.} Then $f$ receives in \textsf{Step 2} $2 \cdot \frac{1}{26}$ charge, in \textsf{Step 3} $\frac{5}{13}$ charge and never contributes charge. Therefore $\ch_3(f) =   0 + \frac{6}{13 }\ge \frac{6}{13} \cdot 1$.
    \item \emph{$f$ is a 2-triangle.} Then $f$ starts with $1$ charge and contributes in \textsf{Step 2} at most $2 \cdot \frac{1}{26}$ charge. Therefore $\ch_3(f) \ge 1 - \frac{1}{13} = \frac{12}{13} \ge \frac{6}{13} \cdot 2$.
    \item \emph{$f$ is a 3-triangle.} Then $f$ never receives or contributes charge. Thus $\ch_3(f) = 2 \ge \frac{6}{13} \cdot 3$.
    \item \emph{$f$ is a 0-quadrilateral.} Then $f$ starts with $0$ charge and never receives or contributes charge as it cannot be the wedge-neighbor of another face. Therefore $\ch_3(f) = 0 \ge \frac{6}{13} \cdot 0$.
    \item \emph{$f$ is a 1-quadrilateral.} Then $f$ starts with 1 charge. If $f$ contributes in \textsf{Step 3} to less than two 1-triangles, we have $\ch_3(f) \ge 1 - \frac{5}{13} = \frac{8}{13} \ge \frac{6}{13} \cdot 1$. Otherwise, we know that $f$ is bounded by a 5-cycle of planar edges (\cref{fig:density-2-planar-proof-c}). Here, charges do not change in \textsf{Step 4}, but we can find $\frac{3}{13}$ charge from the excesses of 2-triangles in this 5-cycle and move that to $f$ in \textsf{Step 5}. Therefore, we have $\ch'(f) = 1 - 2 \cdot \frac{5}{13} + \frac{3}{13} = \frac{6}{13} \ge \frac{6}{13} \cdot 1$.
    \item \emph{$f$ is a 2-quadrilateral.} Then $f$ has one wedge-neighbor, to which it contributes either $\frac{1}{3}$ charge in \textsf{Step 1} or $\frac{5}{13}$ charge in \textsf{Step 3}. So we have $\ch_3(f) \ge 2 - \frac{5}{13} = \frac{21}{13} \ge \frac{6}{13} \cdot 2$ 
    \item \emph{$f$ is a 0-pentagon.} Note that all wedge-neighbors of $f$ are 1-triangles or 2-quadrilaterals, as otherwise there would be an edge with three crossings or a face with two real vertices that are not connected by an edge. If $f$ contributes to five 1-triangles in \textsf{Step 3}, then we would have an $F^2_5$ configuration, which is forbidden.
    Otherwise, at least one 2-quadrilateral contributes its excess of $\frac{14}{13}$ to $f$ in \textsf{Step 4} (see \cref{fig:density-2-planar-proof-d}). Therefore we have $\ch_4(f) \ge 1 + \frac{14}{13} - 4 \cdot \frac{5}{13} = \frac{7}{13}\ge \frac{6}{13} \cdot 0$.
    \item \emph{$f$ is a 1-pentagon or a 2-pentagon resp.} Then $f$ contributes to at most three or two 1-triangles resp.~in \textsf{Step 3}. Therefore, we have $\ch_3(f) \ge 2 - 3 \cdot \frac{5}{13} = \frac{11}{13} \ge \frac{6}{13} \cdot 1$ resp. $\ch_3(f) \ge 3 - 2 \cdot \frac{5}{13} = \frac{29}{13} \ge \frac{6}{13} \cdot 2$.
    \item \emph{$f$ is a 0-hexagon.} If $f$ contributes to six 1-triangles in \textsf{Step 3}, then we would have an $F^2_6$ configuration, which is forbidden. Otherwise, we have $\ch_3(f) \ge 2 - 5 \cdot \frac{5}{13} = \frac{1}{13} \ge \frac{6}{13} \cdot 0$.
    \item \emph{$f$ is a 1-hexagon resp. 2-hexagon.} Then $f$ contributes to at most four resp. three 1-triangles in \textsf{Step 3} and we have $\ch_3(f) \ge 3 - 4 \cdot \frac{5}{13} = \frac{19}{13} \ge \frac{6}{13} \cdot 2$.
    \item \emph{$f$ is a face with $\vert f \vert \ge 7$.} Then $f$ may contribute charge to at most $\vert f \vert$ wedge-neighbors in \textsf{Step 3}. Therefore $\ch_3(f) \ge \vert f \vert + \vert V(f) \vert - 4 - \frac{5}{13} \cdot \vert f \vert \ge \frac{8}{13} \cdot 7 + \vert V(f) \vert - 4 \ge \frac{6}{13} \vert V(f) \vert$.
\end{itemize}
Therefore, all faces $f \in F'$ are satisfied, which proves the proposition.
\end{proof}

Combining \cref{pro:2-planar-charging} and \cref{eq:modified-charge}, $m \le 2 \cdot \frac{13}{6}(n-2)$ is implied, as claimed.

For drawings, where $F^2_6$ configurations are allowed, we can use similar discharging steps to prove the bound of $4.5(n-2)$ on the number of edges. Here we set $\alpha = \frac{4}{9}$, and therefore 1-triangles can receive $\frac{1}{18}$ charge from both its 1-neighbors each in \textsf{Step 2} without creating a demand for any 2-triangles. Therefore, faces have to contribute in \textsf{Step 3} only $\frac{1}{3}$ charge to satisfy all 1-triangles.
Now let $f$ be a 0-hexagon that is the wedge-neighbor of six 1-triangles. Starting with 2 charge, it contributes at most $6 \cdot \frac{1}{3}$ in \textsf{Step 3}, and therefore ends with $0 \ge \frac{4}{9}\cdot 0$ charge. For all other faces we still have enough charge with the same analysis as above.

Therefore, there exists a function $\ch'(\cdot)$ satisfying $\ch'(f) \ge \frac{4}{9} \vert V(f) \vert$ for all $f \in F'$, while the total amount of charge is still $4n-8$. By \cref{eq:modified-charge} we get $m \le 2 \cdot \frac{9}{4}(n-2)$. 
\end{proof}

\subsection{Proof and Discharging for Theorem \ref{th:density-3-planar-forbidden-configuration}}
\thdensitythreeplanar*

\begin{proof}
Let $D$ be a 3-planar $F^3_6$-free drawing that is maximally-dense-crossing-minimal.
As in the proof of \cref{th:density-2-planar-forbidden-configuration}, we assign the initial charges $\ch(f)$ to the faces of $M(D)$ and redistribute them to achieve a function $\ch'(\cdot)$.
The discharging takes place in seven steps:
\begin{itemize}
    \item \textsf{Step 1}: Each 0-triangle receives 1 charge from each 0-neighbor that is a 2-quadrilateral.
    \item \textsf{Step 2}: Each 0-triangle with a demand receives $\frac{1}{3}$ charge from all wedge-neighbors.
    \item \textsf{Step 3}: Each 2-triangle distributes its excess equally over all 1-neighbors that are 1-triangles.
    \item \textsf{Step 4}: Each 1-triangle receives its demand from its wedge-neighbor.
    \item \textsf{Step 5}: Each face distributes its excess equally over the wedge-neighbors that are 0-pentagons, but at most $0.3$ to each of them, and keeps the rest.
    \item \textsf{Step 6}: Each face distributes its excess equally over all vertex-neighbors that are 0-quadrilaterals or 0-pentagons.
    The 0-quadrilaterals distribute this charge equally over their 0-neighbors that have a demand.
    \item \textsf{Step 7}: For each face $f$, let $\mathcal{C}(f)$ be the inclusion-minimal planar cycle of $D$ enclosing $f$ (i.e. the planar cycle that does not contain other planar edges). Then $f$ distributes its excess equally over those faces that lie inside $\mathcal{C}(f)$ and have a demand.
\end{itemize}
Again, we denote by $\ch_i(\cdot)$ the charges after the $i$-th step and by $\ch'(\cdot)$ the final charges. 
Our goal is to show $\ch'(f) \ge 0.4 \vert V(f) \vert$ for all faces $f \in F'$. Note that this is already implied by $\ch_4(f) \ge 0.4 \vert V(f) \vert$, as in \textsf{Step 5-7} faces contribute only their excesses.
We structure the proof into several propositions, collecting statements about the discharging steps.

\begin{proposition}\label{pro:steps-a}
    After \textsf{Step 2}, 0-triangles are and remain satisfied.
\end{proposition}

\begin{proof}
    Let $f$ be a 0-triangle. We have $\ch(f) = -1$. If $f$ receives in \textsf{Step 1} charge, then $\ch_1(f) = 0$. Otherwise, $f$ receives $3 \cdot \frac{1}{3}$ charge in \textsf{Step 2}, so $\ch_2(f)=0$. 0-triangles do not contribute charge in \textsf{Step 3-4}, since they are not wedge-neighbors of 1-triangles. Therefore, $\ch'(f) \ge 0.4 \cdot 0$ holds.
\end{proof}

\begin{proposition}\label{pro:steps-b}
   In \textsf{Step 1-2}, 0-faces contribute no charge.
\end{proposition}

\begin{proof}
    No faces except 2-quadrilaterals contribute charge in \textsf{Step 1}, so we consider only \textsf{Step 2}. Assume that a 0-face $f_0$ contributes charge to a 0-triangle $f$ in \textsf{Step 2}, and $f_0$ and~$f$ are therefore 0-neighbors at an edge $e_0$. Let $e_1, e_2$ be the other edges of $f$ and $f_1, f_2$ the 0-neighbors at these edges (see \cref{fig:easy-faces-a}).
    Since $f_0$ is a 0-face, it is incident to two crossings each with $\overline{e}_1$ and $\overline{e}_2$ and these edges also cross each other at $f$. Therefore $\overline{e}_1$ and $\overline{e}_2$ have already three crossings and end at $f_1$ resp.\ $f_2$. The edge $\overline{e}_0$ ends also at one of $f_1$ or $f_2$, as otherwise it would have four crossings. W.l.o.g.\ $e_0$ ends at $f_1$ and by \cref{pro:charging-helper} $f_1$~is a 2-quadrilateral. Hence, $\ch_1(f) \ge 0.4 \cdot \vert V(f) \vert $, contradicting that $f$ receives charge later.
    \begin{figure}
        \begin{subfigure}[b]{0.24\linewidth}
        \center
        \includegraphics[width=\textwidth, page=1]{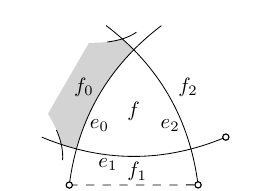}
        \subcaption{}
        \label{fig:easy-faces-a}
        \end{subfigure}  
        \hfill
        \begin{subfigure}[b]{0.24\linewidth}
        \center
        \includegraphics[width=\textwidth, page=2]{3planarCases.pdf}
        \subcaption{}
        \label{fig:easy-faces-b}
        \end{subfigure} 
        \hfill
        \begin{subfigure}[b]{0.24\linewidth}
        \center
        \includegraphics[width=\textwidth, page=3]{3planarCases.pdf}
        \subcaption{}
        \label{fig:easy-faces-c}
        \end{subfigure}
        \hfill
        \begin{subfigure}[b]{0.24\linewidth}
        \center
        \includegraphics[width=\textwidth, page=4]{3planarCases.pdf}
        \subcaption{}
        \label{fig:easy-faces-d}
        \end{subfigure} 
        \hfill
        \caption{Illustrations for the proofs of \cref{pro:steps-b}, \cref{pro:steps-c} and \cref{pro:steps-d}.}
        \label{fig:easy-faces}
    \end{figure}
\end{proof}

\begin{proposition}\label{pro:steps-c}
    After \textsf{Step 3}, 1-triangles have a demand of at most $0.3$ charge.
\end{proposition}

\begin{proof}
    Let $f$ be a 1-triangle with the real vertex $v$ and the 0-edge $e$. Let further $f_1, f_2$ be the 1-neighbors of $f$ (see \cref{fig:easy-faces-b}). Then $\overline{e}$ ends at one of $f_1$ and $f_2$, as otherwise it would have more than three crossings. W.l.o.g. let $f_1$ be that face with the vertex $v'$ to which $\overline{e}$ is incident. Then by \cref{pro:charging-helper} the edge $vv'$ exists and $f_1$ is a 2-triangle. Therefore, $f_1$ starts with 1 charge and has an initial excess of $0.2$. Thus, $f$ receives $0.1$ charge in \textsf{Step 3}. We have $\ch_3(f) = 0.1$, which is equivalent to a demand of $0.3$.
\end{proof}

\begin{proposition}\label{pro:steps-d}
    After \textsf{Step 4}, all 1-quadrilaterals are satisfied.
\end{proposition}

\begin{proof}
    Let $f$ be a 1-quadrilateral. We have $\ch(f)=1$ and $f$ contributes charge only in \textsf{Step 2} and \textsf{Step 4}. If $f$ contributes to at most one wedge-neighbor or to two 1-triangles, then $\ch_4(f) \ge 1-0.6 \ge 1 \cdot 0.4$. Otherwise, $f$ contributes either to two wedge-neighbors that are both 0-triangles or to one 0-triangle and one 1-triangle.
    In the first case, both 0-triangles are already satisfied after \textsf{Step 1}, as they have wedge-neighbors that are 2-quadrilaterals (see \cref{fig:easy-faces-c}). In the second case, $f$ contributes not more than $0.2$ charge to the 1-triangle $f'$, because one of its 1-neighbors is a 2-triangle contributing its excess of $0.2$ charge only to $f'$ in \textsf{Step 3} (see \cref{fig:easy-faces-d}). Therefore, we have $\ch_4(f) \ge 1- \frac{1}{3}-0.2 \ge 1 \cdot 0.4$.
\end{proof}

\begin{proposition}\label{pro:steps-e}
    After \textsf{Step 4}, all faces are and remain satisfied that are not 0-pentagons that are the wedge-neighbor of four or five 1-triangles.
\end{proposition}

\begin{proof}
    Note again that for a face $f$ the charge $\ch_4(f) \ge 0.4 \vert V(f) \vert $ implies already that it has no demand in \textsf{Step 5-7}, since faces only there contribute their excesses.
    
    To see that 0-triangles and 1-quadrilaterals are satisfied, we refer to Propositions \ref{pro:steps-a} and \ref{pro:steps-d}. 1-triangles are satisfied by definition of \textsf{Step 4}. Remember that only 3-triangles and $r$-$s$-gons with $r\le 2, s \ge 3$ can exist by \cref{pro:charging-helper}. Now we discuss the other cases:
    \begin{itemize}
        \item \emph{$f$ is a 2-triangle.} We start with $\ch(f) = 1$. As only wedge-neighbors contribute in \textsf{Step 1-2} and \textsf{Step 4} and $f$ cannot be a wedge-neighbor of another face, the only critical step is \textsf{Step 3}. Here, $f$ contributes in total at most its excess of $0.2$ charge, and therefore $\ch_4(f) \ge 1-0.2 \ge 2 \cdot 0.4$.
        \item \emph{$f$ is a 3-triangle.} We start with $\ch(f) = 2$ and $f$ never contributes charge. It follows that $\ch_4(f) = 2 \ge 3 \cdot 0.4$ holds.
        \item \emph{$f$ is a 0-quadrilateral.} Again, $f$ never contributes charge, and therefore $\ch(f) = \ch_4(f) = 0 \ge 0 \cdot 0.4$ holds.
        \item \emph{$f$ is a 2-quadrilateral.} We start with $\ch(f) = 2$. Note that $f$ contributes only once as it has only one wedge-neighbor, and therefore we have $\ch_4(f) \ge 2-1 \ge 2 \cdot 0.4$.
        \item \emph{$f$ is a 0-pentagon with at most three wedge-neighbors that are 1-triangles.} We have $\ch(f) = 1$ and $f$ contributes only to three faces. With \cref{pro:steps-b} and \cref{pro:steps-c} $\ch_4(f) \ge 1- 3 \cdot 0.3 \ge 0 \cdot 0.4$ follows.
        \item \emph{$f$ is a 1-pentagon or a 2-pentagon.} $f$ starts with $\ch(f) \ge 2$ and we have $\ch_4(f) \ge 2-3\cdot\frac{1}{3} \ge 2 \cdot 0.4$.
        \item \emph{$f$ is a face with $\vert f \vert \ge 6$.} Then $f$ may contribute to at most $\vert f \vert$ wedge-neighbors charge. Therefore, we have $\ch_4(f) \ge \vert f \vert + \vert V(f) \vert - 4 - \frac{1}{3} \cdot \vert f \vert \ge \vert V(f)\vert$.
    \end{itemize}
\end{proof}

It remains to prove that 0-pentagons with four or five wedge-neighbors that are 1-triangles have at least zero charge after \textsf{Step 7}. We show this by the following four propositions, which we only state here; the proofs can be found in \cref{app:charging}.

\begin{restatable}{proposition}{prostepf}\label{pro:steps-f}
    In \textsf{Step 5}, each 0-pentagon receives $0.3$ charge from all wedge-neighbors that are not 1-triangles, 0-triangles or 0-pentagons.
\end{restatable}

\begin{restatable}{proposition}{prostepg}\label{pro:steps-g}
    In \textsf{Step 6}, each 1-face and 2-face $f$ with $\vert f \vert \ge 5$ and each 0-face $f$ with $\vert f \vert \ge 7$ contributes at least $0.4$ charge to the vertex-neighbors that are 0-quadrilaterals or 0-pentagons.
\end{restatable}

\begin{restatable}{proposition}{prosteph}\label{pro:steps-h}
     After \textsf{Step 7}, all 0-pentagons that are the wedge-neighbor of four 1-triangles are satisfied.
\end{restatable}

\begin{restatable}{proposition}{prostepi}\label{pro:steps-i}
    After \textsf{Step 7}, all 0-pentagons that are the wedge-neighbor of five 1-triangles are satisfied.
\end{restatable}

By Propositions \ref{pro:steps-e}, \ref{pro:steps-h} and \ref{pro:steps-i} $\ch'(f) \ge 0.4 \cdot \vert V(f) \vert$ holds for all faces $f \in F'$. Since~charge is only moved, its total amount is still $4n-8$ and \cref{eq:modified-charge} implies $m \le \frac{2}{0.4}(n-2)$.
\end{proof}

\section{Proof of Theorem \ref{th:linear-bounds}}\label{sec:linear-bounds}

In this section, we present the proof of \cref{th:linear-bounds} that shows how to use the earlier stated observations and theorems and leads to a better bound for the Crossing Lemma. 

\thlinearbounds*
\begin{proof}
We start proving the bound in (a). If $m \le 5(n-2)$, then the bound follows from the linear bound $\cross(G) \ge \frac{7}{3}m - \frac{25}{3}(n-2)$ \cite{DBLP:journals/dcg/PachRTT06}. So assume $m > 5(n-2)$ and let $D$ be a crossing-minimal drawing of $G$.
From $D$, we iteratively remove the edge with the most  crossings until $5(n-2)$ edges are left. In particular, as long as the maximum number of crossings is three, we always remove an edge from an $F^3_6$ configuration. By \cref{th:density-3-planar-forbidden-configuration}, we stop latest, when there are no $F^3_6$ configurations. 
By this process, edges are iteratively deleted until we reach $5(n-2)$ edges, as following:
\begin{itemize}
    \item $m_{5+}$ edges with five or more crossings -- denote the resulting drawing by $D_4$,
    \item then $m_4$ edges with four crossings -- denote the resulting drawing by $D_3$ and the set of edges deleted in this step by $E_4$,
    \item then $m_3$ edges with three crossings from $F^3_6$ configurations -- denote the resulting drawing by $D_{3-}$.
\end{itemize}
Note that $m_4$ or $m_3$ could be zero in the case that we reached $5(n-2)$ already during step (1) or (2).
Afterwards we have $m_3$ edge-disjoint $F^3_6$ configurations with a missing edge in $D_{3-}$.
\begin{figure}
    \center
    \includegraphics[width=0.35\textwidth, page=1]{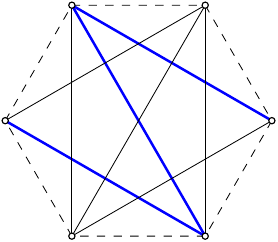}  
    \hfil
    \caption{Three independent edges (blue) with three crossings in an $F^3_6$ configuration. In $D_3$ one of them is already deleted, in $D_{3-}$ also the other two.}
    \label{fig:linear-bound}
\end{figure}
So we are able to find $2m_3$ more independent edges with three crossings and delete them (see \cref{fig:linear-bound}). Continue the deletion process by still removing the edge with the most crossings until this edge no longer has three or more crossings; we denote the number of these deleted edges by $m_{3-}$. Call the achieved drawing $D_2$. By applying the linear bound from \cite{DBLP:journals/dcg/PachRTT06} again, we have
\begin{align}
    \cross(G) &\ge [5m_{5+} + 4m_4 + 3m_3 ] + [2 \cdot 3m_3 + 3m_{3-}] + \left[\frac{7}{3}(5(n-2)-2m_3-m_{3-}) - \frac{25}{3}(n-2)\right]\notag \\
    &= 5m_{5+} + 4m_4 + \frac{13}{3}m_3 + \frac{2}{3}m_{3-} + \frac{10}{3}(n-2).\label{eq:edges-splitted}
\end{align}
As all values are non-negative, it is not hard to see that this is at least \[\ge 4(m_{5+}+m_4+m_3+5(n-2))- \frac{50}{3}(n-2) = 4m-\frac{50}{3}(n-2).\]
For the better bound of $\cross(G) \ge \frac{37}{9}m - \frac{155}{9} (n-2)$ we have to elaborate on the value $m_4$, as there was no slack in the last inequality.

As a preparation, we first consider the structure of $D_2$.
Let $c_{pent}$ be the number of $F^2_5$ configurations and $c_{hex}$ the number of $F^2_6$ configurations in $D_2$. Let further $E_0$ be the set of crossing-free edges on the boundary of the forbidden configurations in $D_3$ resp. $D_2$ that do not exist in $D_2$, and therefore may be added.
We denote $\vert E_0\vert = m_0$ and state the following; the proof is in \cref{app:linear-bound}.

\begin{restatable}{proposition}{prooptface}\label{pro:optimal-faces}
With the notation above, $c_{pent} + c_{hex} \ge \frac{2}{3}(n-2) - \frac{4}{3}m_3 - m_{3-}+m_0$.
\end{restatable}

Next, we show how to limit the number of the edges of $E_4$, i.e., the deleted edges that were accounted with four crossings in $D_4$. 
For that, we introduce a triangulation $\mathcal{H}$ on the vertices of $G$ such that has the following properties. Edges with property $(i)$ are contained in subset $S_i$ for $1\leq i \leq 5$:

\begin{itemize}
\item[(1)] $\mathcal{H}$ contains 
the boundary of every $F^3_6$ configuration in $D_3$
\item[(2)] $\mathcal{H}$ contains 
 the boundary of every $F^2_5$ and $F^2_6$ configuration in $D_2$
\item[(3)] $\mathcal{H}$ contains 
  every edge in $E_4$ that lies completely outside of these forbidden configurations
 Furthermore 
\item[(4)] Let $uu'$ be a boundary edge of the forbidden configuration $P$ that exists in $D_4$ and is crossed by at least one edge $e \in E_4$. 
We consider the crossing edge, say $e = ab$ that has the crossing $c$ with $uu'$ closest to $u$. Consider the two segments $(a,c)$ and $(c,b)$ of $e$, such that $(a,c)$ is completely outside of the forbidden configuration $P$. Then $\mathcal{H}$ contains the triangle $t$, which is adjacent to $uu'$ and consists of the edge $uu'$, as second edge, we define the edge that closely follows the two segments $(u,c)$ and $(c,a)$, and as the third edge of $t$, we take the edge that closely follows the two segments $(a,c)$ and $(c,u')$ (see \cref{fig:definition-H-a}). We call $t$ the \emph{charging triangle} of $uu'$ and $(a,c)$ the \emph{pillar} of $t$.
\item[(5)]  Let $\tilde{e}_1$, $\tilde{e}_2$ be the edge-segments enclosing the 2-triangle next to $uu'$ in the forbidden configuration.
For the case that a $\tilde{e}_i$ with $i \in \{1,2\} $ is crossed by two edges $e,e' \in E_4$, let $\tilde{e}_1$ be incident to $u$, let $a,a'$ be the endpoints of $e,e'$ that lie outside of the forbidden configuration and let $c'$ be the crossing of the edges $au'$ and $e'$. If $a\ne a'$, then $\mathcal{H}$ contains the triangle $t'$ defined by the edges $u'a$, $aa'$ and $au'$, where $aa'$ closely follows the edge segments $(ac')$ and $(c'a')$ of $ua'$ and $e'$ and $au'$ closely follows the edge segments $(a'c')$ and $(c'u')$ (see \cref{fig:definition-H-b}). We call $t'$ the \emph{charging by-triangle} of $uu'$ and $(a',c')$ the \emph{pillar} of $t$.
\end{itemize}
Note that pillars do not cross a boundary edge as this would imply that the corresponding edges $e,e' \in E_4$ have more than four crossings.

\begin{figure}
    \begin{subfigure}[b]{0.35\linewidth}
    \center
    \hfil
    \includegraphics[width=\textwidth, page=4]{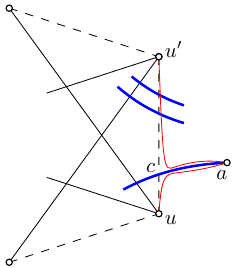}
    \subcaption{}
    \label{fig:definition-H-a}
    \end{subfigure}  
    \hfil
    \begin{subfigure}[b]{0.35\linewidth}
    \center
    \includegraphics[width=\textwidth, page=7]{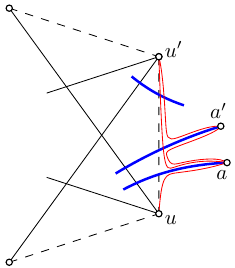}
    \subcaption{}
    \label{fig:definition-H-b}
    \end{subfigure}
    \hfil
    \caption{The construction of a charging triangle (a) and a charging by-triangle (b) of $uu'$.}
    \label{fig:definition-H}
\end{figure}

\begin{claim}
    A triangulation $\mathcal{H}$ with the properties (1)-(5) always exist.
\end{claim}

\begin{proof}
Let $\mathcal{H}$ be a triangulation constructed with the properties (1)-(5).
We need to show that the edges of $S_1,...,S_5$ do not cross each other. The boundary edges of $F^3_6$ configurations in $D_3$ and of  $F^2_5$ and $F^2_6$ configurations in $D_2$ are planar and therefore edges of $S_1$ and $S_2$ do not cross each other. As edges of $E_4$ do not cross each other, edges of $S_3$ do not cross other edges of $S_3$ and by construction they do not cross edges of $S_1, S_2$.

The edges of $S_4$ consist of segments $(x,y)$ closely following pillars and segments of boundary edges. Note that every edge crossing $(x,y)$ also crosses the edge segment $(x,y)^*$ closely followed. Therefore, $(x,y)$ crossed by an edge of $S_1$ or $S_2$ would imply that either boundary edges cross each other or a pillar crosses a boundary edge, both contradictions. If $(x,y)$ is crossed by an edge of $S_3$, that would imply that edges of $E_4$ cross each other or an edge of $S_3$ crosses a boundary edge, both contradictions.
For the case that two edge $(x_1,y_1), (x_2,y_2)$ of $S_4$ cross each other, it is sufficient to observe that then $(x_1,y_1)^*, (x_2,y_2)^*$ cross and again boundary edges and pillars can not cross.

For the edges of $S_5$ we argue as for the edges of $S_4$. Here, the edges are following pillars or edge segments of $S_4$. Again, an edge segment $(x,y)$ crossed by an edge $e$ would imply that $(x,y)^*$ is crossed by $e$. Thus, an edge of $S_5$ crossed by an edge of $S_1,...,S_5$ would imply that edges of $S_1,...S_4$ cross each other or a pillar, which is a contradiction.

\end{proof}    

\begin{restatable}{proposition}{promfour}\label{pro:m_4andm_3}
Let $\mathcal{H}' \subseteq \mathcal{H}$ be the set of triangles that do not belong to the forbidden configurations and let $c_\triangle = \vert \mathcal{H}' \vert$. Then $m_4 \le m_3 + c_{hex} + 4 m_0 + 2 c_\triangle$.
\end{restatable}

The proof can be found in \cref{app:linear-bound}.
Combining the results, we can finish the first part of the proof. \cref{pro:optimal-faces} implies
\begin{align*}
	 c_\triangle \le 2(n-2) - 4m_3 - \left[\frac{2}{3}(n-2) - \frac{4}{3}m_3 - m_{3-} +m_0\right] \cdot 3 - c_{hex} =  3 m_{3-} - 3m_0 - c_{hex},
\end{align*}
because the total number of triangles is $2(n-2)$ and a pentagon resp. hexagon contains three resp. four triangles. Together with \cref{pro:m_4andm_3}, this gives
\begin{align*}
    m_4 \le m_3 + c_{hex} + 4 m_0 + 2 (3 m_{3-} - 3m_0 - c_{hex}) \le 2m_3 + 6m_{3-}.
\end{align*}
Multiplying this term by $\frac{1}{9}$ and adding it to \cref{eq:edges-splitted}, we get as desired
\begin{align*}
    \cross(G) &\ge 5m_{5+} + 4m_4 + \frac{13}{3}m_3 + \frac{2}{3}m_{3-} + \frac{10}{3}(n-2) + \frac{m_4-(2m_3 + 6m_{3-})}{9}\\
    &\ge \frac{37}{9}(m_{5+}+m_4+m_3) + \frac{10}{3}(n-2)\\
    &= \frac{37}{9}(m_{5+}+m_4+m_3+5(n-2)) - \frac{155}{9}(n-2).
\end{align*}

\medskip

For the bound in (b) see the following: If $m \le 6(n-2)$, then we can apply the bound of (a). So let be $m > 6(n-2)$. Iteratively delete the edge with the most crossings in a crossing-minimal drawing $D$ until $6(n-2)$ edges are left; these edges have at least five crossings, as the density of 4-planar graphs is $\le 6(n-2)$ \cite{DBLP:journals/comgeo/Ackerman19}. With the bound in (a), this~implies
\begin{align*}
    \cross(G) & \ge 5(m-6(n-2)) + \frac{37}{9} \cdot 6(n-2) - \frac{155}{9}(n-2) = 5m - \frac{203}{9}(n-2).
\end{align*}
\end{proof}

\section{Discussion}
We have improved the leading constant of the lower bound for the crossing number of a given graph $G$. Although this improvement does not seem to be too impressive at first sight, we worked out some interesting observations for drawings with a limited number of crossings~per edge.
This leads to further improvements, conjectures and suggestions for future research.

In particular, we have improved for $m > 5(n-2)$ the lower bound of the crossing number, unfortunately we did not reach tightness. We confirm the conjecture by 
\cite{DBLP:journals/dcg/PachRTT06} that $\cross(G) \ge \frac{25}{6}m-\frac{35}{2}(n-2)$ holds. The corresponding upper bound can be obtained by a construction where the plane subgraph consists only of pentagonal and hexagonal faces \cite{DBLP:journals/dcg/PachRTT06}.

Our improvement compared to the previous version \cite{LIPIcs:BuengenerKaufmann} comes from \cref{pro:m_4andm_3}, where we were able to replace $m_4 \le m_3 + c_{hex} + 4m_0 + 4 c_\triangle$ by $m_4 \le m_3 + c_{hex} + 4m_0 + 2 c_\triangle$. Although it seems not possible to get the bound $m_4 \le m_3 + c_{hex} + 4m_0 + \frac43 c_\triangle$, from which the conjecture would follow, we think that our general approach in the proof is convenient.

Applying our technique to 4-planar drawings might show that these drawings without full
hexagons $F^4_6$ have density $\le 5.5(n-2)$. This would provide a characterization of optimal 4-planar graphs, which is a well-known open problem.
Further, we can look at 5-planar graphs, a class that has been considered as too complex for actual research. Just applying \cref{cor:density-k-planar} improves the current known density bound from $8.52n$ to $8.3n$.

It seems to be worthwhile to apply the idea to bipartite graphs to obtain improvements of the Crossing Lemma. 
Here, the corresponding linear bound $\cross(G) \ge 3m - \frac{17}{2}n+19$ used in the current proof in \cite{DBLP:journals/corr/abs-1712-09855} is not tight.

Furthermore, we have indicated a way how to obtain the exact density bound of optimal simple 3-planar graphs. Note that we only did one step in this direction.

\bibliographystyle{plainurl}
\bibliography{biblio}

\newpage
\appendix

\section{Details for Section \ref{sec:charging}}\label{app:charging}

\prostepf*

\begin{proof}
    Note that in the calculations of \cref{pro:steps-e}, we assumed for all faces that are possibly a wedge-neighbor of a 0-pentagon except 0-triangles, 1-triangles, 1-quadrilaterals and 0-pentagons that they give at least $0.3$ charge to all wedge-neighbors.
    If such a face $f$ has a wedge-neighbor that is a 0-pentagon, then it did not contribute charge to it in \textsf{Step 1-4}, and therefore has $0.3$ charge left for it in \textsf{Step 5}. The only critical case is a 1-quadrilateral $f$ with a 0-pentagon and a 0-triangle $f'$ as wedge-neighbors.
    Observe that in this case $\ch_1(f') = 0$ already, because there is a 2-quadrilateral next to $f'$ as in \cref{fig:easy-faces-c}, and therefore $\ch_4(f) = 1$. Thus, $f$ can contribute $0.3$ charge to the 0-pentagon.
\end{proof}

\prostepg*

\begin{proof}
    Let $e$ be a 0-edge of a face $f$ incident to the crossings $x$ and $y$. Let $f_1$ be the vertex-neighbor of $f$ at $x$ and $f_2$ the vertex-neighbor at $y$. If $f_1$ and $f_2$ are 0-faces, then $\overline{e}$ has more than three crossings, a contradiction. Therefore, no face can contribute charge through two consecutive crossings on its boundary in \textsf{Step 6}.
    For a face $f$, this implies that it can contribute to at most $\left\lfloor \frac{\vert f \vert}{2}\right\rfloor$ vertex-neighbors in this step.
    
    Now we distinguish different cases for the face $f$ that might contribute to vertex-neighbors. We start with the case that $f$ is a 0-face. Here, after \textsf{Step 5}, $f$ has an excess of
    \begin{align*}
        \ch_5(f) &\ge \vert f \vert - 4 - \vert f \vert \cdot 0.3 = 0.7 \vert f \vert - 4,
    \end{align*}
    which is at least $0.4 \cdot \left\lfloor \frac{\vert f \vert}{2}\right\rfloor$ for $\vert f \vert \ge 8$ and therefore enough. If $f$ is a 0-heptagon, then, by the inequality above, there is enough charge if $f$ contributes to at most two vertex-neighbors in \textsf{Step 6}. So assume that it contributes to three vertex-neighbors. 
    \begin{figure}
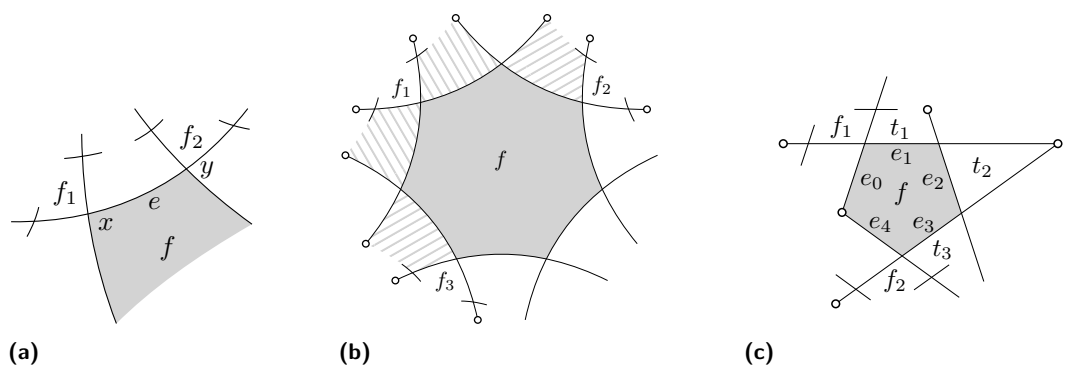

        \begin{subfigure}[b]{0.23\linewidth}
        \center
        \includegraphics[width=\textwidth, page=5]{3planarCases.pdf}
        \subcaption{}
        \label{fig:vertex-neighbors-a}
        \end{subfigure}  
        \hfill
        \begin{subfigure}[b]{0.3\linewidth}
        \center
        \includegraphics[width=\textwidth, page=6]{3planarCases.pdf}
        \subcaption{}
        \label{fig:vertex-neighbors-b}
        \end{subfigure}  
        \hfill
        \begin{subfigure}[b]{0.3\linewidth}
        \center
        \includegraphics[width=\textwidth, page=7]{3planarCases.pdf}
        \subcaption{}
        \label{fig:vertex-neighbors-c}
        \end{subfigure} 
        \hfill
        \caption{(a) No face contributes to two consecutive vertex-neighbors in \textsf{Step 6}. (b) If a 0-heptagon $f$ contributes to three vertex-neighbors $f_1, f_2, f_3$ in \textsf{Step 6}, then it contributes not to all its seven wedge-neighbors in \textsf{Step 1-5}. (c) A 1-pentagon $f$ contributing to all three wedge-neighbors in \textsf{Step 1-5} and to two vertex-neighbors in \textsf{Step 6} leads to a contradiction.}
        \label{fig:vertex-neighbors}
    \end{figure}
    It is not hard to see that there are wedge-neighbors of $f$ with $\vert f \vert \ge 4$ and $\vert V(f) \vert \ge 1$ (see \cref{fig:vertex-neighbors-a}), so $f$ contributes in \textsf{Step 1-5} to at most six faces. Now we have $\ch_5(f) \ge 3 - 6 \cdot 0.3 = 1.2$, which is sufficient to give three vertex-neighbors $0.4$ charge each in \textsf{Step 6}.
    
    Let $f$ now be a 1-face. Then $f$ has an excess of
    \begin{equation*}
        \ch_5(f)-0.4 \ge \vert f \vert + 1 - 4 - (\vert f \vert - 2) \cdot \frac{1}{3} -0.4 = \frac{2}{3} \vert f \vert -3.4 +\frac{2}{3},
    \end{equation*}
    which is at least $0.4 \cdot \left\lfloor \frac{\vert f \vert}{2}\right\rfloor$ if $\vert f \vert \ge 6$. If $f$ is a 1-pentagon, then, by the inequality above, there is enough charge if $f$ contributes to at most one vertex-neighbor in \textsf{Step 6}. So assume the opposite, i.e., $f$ contributes to two vertex-neighbors $f_1, f_2$ in \textsf{Step 6}.
    If $f$ contributes in \textsf{Step 1-5} to only one or two wedge-neighbors, then its excess after \textsf{Step 5} is at least $1.6-\frac{2}{3}$ charge and therefore sufficient, so assume this is not the case either.
    
    Walking along the boundary of $f$, let $e_0$ be a 1-edge of $f$, let $e_1, e_2, e_3$ be the 0-edges and let $e_4$ be the other 1-edge of $f$. Let further $t_i$ be the wedge-neighbor of $f$ at $e_i$ for $i \in \{1,2,3\}$.
    W.l.o.g. $f_1$ lies at the crossing of $e_0$, and therefore the face $t_2$ is a 1-triangle, as otherwise $f$ would not contribute charge to $t_2$ in \textsf{Step 1-5} (see \cref{fig:vertex-neighbors-b}). Therefore, $f_2$ lies at the crossing of $e_4$. Note that $\overline{e}_2$ ends at $t_1$ or $t_3$, say w.l.o.g. at $t_1$. But then $\vert t_1 \vert \ge 4$ and $\vert V(t_1) \vert \ge 1$, so $f$ does not contribute charge to $t_1$ in \textsf{Step 1-5}, a contradiction to our assumption.
    Therefore, 1-pentagons can contribute $0.4$ charge to the desired vertex-neighbors.
    
    The last case is that $f$ is a 2-face. Then $f$ has an excess of
    \begin{equation*}
        \ch_5(f)-0.8 \ge \vert f \vert + 2 - 4 - (\vert f \vert - 3) \cdot \frac{1}{3} -0.8 \ge \frac{2}{4} \vert f \vert -1.8,
    \end{equation*}
    which is at least $0.4 \cdot \left\lfloor \frac{\vert f \vert}{2}\right\rfloor$ for $\vert f \vert \ge 5$.
\end{proof}

\prosteph*

\begin{proof}
    We introduce a notation for the edges and faces at a 0-pentagon $f$. Let $e_i, i \in \{0,...,4\}$ be the edges forming the boundary of $f$, so that $e_i$ and $e_{(i+1\hspace{-0.15cm} \mod 5)}$ have a crossing at $f$. Further we denote by $t_i$ the wedge-neighbor of $f$ at $e_i$ and by $f_i$ the vertex-neighbor of $f$ at the crossing of $e_i$ and $e_{(i+1\hspace{-0.15cm} \mod 5)}$.
    
    Let $f$ be a 0-pentagon with four wedge-neighbors that are 1-triangles. So we have $\ch_4(f) \ge 1- 4 \cdot 0.4 = -0.2$. Let w.l.o.g. $t_0$ be the wedge-neighbor of $f$ that is not a 1-triangle. If $t_0$ is not a 0-triangle or 0-pentagon, then it contributes, by \cref{pro:steps-f}, $0.3$ charge to $f$ in \textsf{Step 5} and $f$ is satisfied.
    Otherwise, distinguish between the type of the face $t_0$.
    \begin{itemize}
        \item \emph{Case 1: $t_0$ is a 0-triangle.} 
        Observe that $\overline{e}_1$ and $\overline{e}_4$ already have three crossings and $\overline{e}_0$ two crossings. Therefore, $\overline{e}_0$ ends at $f_0$ or $f_4$, say w.l.o.g. $f_0$, so $f_0$ is a 2-quadrilateral (see \cref{fig:1-triangle-case-1-a}).
        \begin{figure}
        \begin{subfigure}[b]{0.32\linewidth}
        \center
        \includegraphics[width=\textwidth, page=8]{3planarCases.pdf}
        \subcaption{}
        \label{fig:1-triangle-case-1-a}
        \end{subfigure}  
        \hfill
        \begin{subfigure}[b]{0.32\linewidth}
        \center
        \includegraphics[width=\textwidth, page=9]{3planarCases.pdf}
        \subcaption{}
        \label{fig:1-triangle-case-1-b}
        \end{subfigure} 
        \hfill
        \begin{subfigure}[b]{0.32\linewidth}
        \center
        \includegraphics[width=\textwidth, page=10]{3planarCases.pdf}
        \subcaption{}
        \label{fig:1-triangle-case-1-c}
        \end{subfigure}
        \hfill
        \caption{Illustrations for the proof of \cref{pro:steps-h} Case 1 and 2.1.}
        \label{fig:1-triangle-case-1}
        \end{figure}
        Then $f_0$ has an excess of $0.2$ after \textsf{Step 5}, as it only contributes in \textsf{Step 1} charge.
        Note that the vertex-neighbors of $f_0$ are $f$ and $f_4$. Since $f_4$ is not a 0-face, $f_0$ contributes its excess of $0.2$ charge in \textsf{Step 6} only to $f$, and therefore $\ch_6(f) \ge 0$.
        
        \item \emph{Case 2: $t_0$ is a 0-pentagon.} 
        Again, $\overline{e}_1$ and $\overline{e}_4$ already have three crossings and $\overline{e}_0$ two crossings. Therefore, $f_2$ is a 2-triangle and also one of $f_1$ and $f_3$, say w.l.o.g. $f_1$ (see \cref{fig:1-triangle-case-1-b}). So we have $\ch_3(t_2) = -0.2$, and therefore $\ch_4(f) \ge 1- 3 \cdot 0.3 - 0.2 \ge -0.1$.
        
        Note that the only face besides $f$ that may receive charge from $t_0$ in \textsf{Step 5} is $f_4$. Therefore, we distinguish two cases:
        \begin{itemize}
            \item \emph{Case 2.1: $f_4$ is a 0-pentagon.} If less than three wedge-neighbors of $t_0$ are 1-triangles, then $\ch_4(t_0) \ge 0.4$ and $f$ receives enough charge in \textsf{Step 5}. If three wedge-neighbors of $t_0$ are 1-triangles, then the 1-triangle at the vertex at which $\overline{e}_1$ ends has $-0.2$ charge after \textsf{Step 3}, as it lies between two 2-triangles (see \cref{fig:1-triangle-case-1-c}). Therefore, we have $\ch_4(t_0) = 1-2\cdot 0.3 -0.2 = 0.2$ and $f$ can receive a half of it in \textsf{Step 5}, which is enough.
            
            \item \emph{Case 2.2: $f_4$ is not a 0-pentagon.} If $\ch_4(t_0) \ge 0.1$, then $f$ receives its missing charge already in \textsf{Step 5}. So assume the opposite, which implies that four wedge-neighbors of $t_0$ are 1-triangles (see \cref{fig:1-triangle-case-1.2.2-a}).
            
            If now $f_3$ is a 2-triangle, then $\ch_3(t_3)= -0.2$ and we have $\ch(f) = 1-  2 \cdot 0.3 - 2 \cdot 0.2 =0$ and $f$ never has a demand. Otherwise, there is an edge $\overline{e}$ crossing $\overline{e}_3$ at $f_3$ that has already three crossings, and therefore $f_3$ is either a 2-quadrilateral or a 1-triangle. In the first case, $f_3$ has an excess of at least $0.2$ and only contributes it to $f$ in \textsf{Step 6}. In the second case, we have a planar cycle of length seven, in which all faces except $f$ are satisfied after \textsf{Step 6}. Here, $f$ receives its demand in \textsf{Step 7} from a 2-quadrilateral that is a vertex-neighbor of $t_0$ (\cref{fig:1-triangle-case-1.2.2-c}). In all cases $f$ is satisfied after \textsf{Step 6}.
        \end{itemize}
    \begin{figure}
        \begin{subfigure}[b]{0.32\linewidth}
        \center
        \includegraphics[width=\textwidth, page=11]{3planarCases.pdf}
        \subcaption{}
        \label{fig:1-triangle-case-1.2.2-a}
        \end{subfigure}
        \hfill
        \begin{subfigure}[b]{0.32\linewidth}
        \center
        \includegraphics[width=\textwidth, page=12]{3planarCases.pdf}
        \subcaption{}
        \label{fig:1-triangle-case-1.2.2-b}
        \end{subfigure} 
        \hfill
        \begin{subfigure}[b]{0.32\linewidth}
        \center
        \includegraphics[width=\textwidth, page=13]{3planarCases.pdf}
        \subcaption{}
        \label{fig:1-triangle-case-1.2.2-c}
        \end{subfigure}
        \hfill
        \caption{Illustrations for the proof of \cref{pro:steps-h} Case 2.2.}
        \label{fig:1-triangle-case-1.2.2}
    \end{figure}
    \end{itemize}    
\end{proof}

\prostepi*

\begin{proof}
    We continue to use the notation introduced in the proof of \cref{pro:steps-h}.
    Let $f$ be a 0-pentagon with five wedge-neighbors that are 1-triangles. We distinguish the number of 0-neighbors of $f$ that are 0-quadrilaterals. Note that at most two such faces can exist next to a 0-pentagon.
    \begin{itemize}
        \item \emph{Case 1: No 0-neighbor of $f$ is a 0-quadrilateral.} Then all five vertex-neighbors are 2-triangles (see \cref{fig:1-triangle-case-2.1-2.2-a}) and we have $\ch'(f) = \ch_4(f) = 1- 0.5 \cdot 2 = 0$.
    \begin{figure}
        \begin{subfigure}[b]{0.3\linewidth}
        \center
        \includegraphics[width=\textwidth, page=14]{3planarCases.pdf}
        \subcaption{}
        \label{fig:1-triangle-case-2.1-2.2-a}
        \end{subfigure}  
        \hfill
        \begin{subfigure}[b]{0.3\linewidth}
        \center
        \includegraphics[width=\textwidth, page=15]{3planarCases.pdf}
        \subcaption{}
        \label{fig:1-triangle-case-2.1-2.2-b}
        \end{subfigure} 
        \hfill
        \begin{subfigure}[b]{0.3\linewidth}
        \center
        \includegraphics[width=\textwidth, page=16]{3planarCases.pdf}
        \subcaption{}
        \label{fig:1-triangle-case-2.1-2.2-c}
        \end{subfigure}
        \hfill
        \caption{Illustrations for the proof of \cref{pro:steps-i} Cases 1 and 2.}
        \label{fig:1-triangle-case-2.1-2.2}
    \end{figure}

        \item \emph{Case 2: Exactly one 0-neighbor of $f$ is a 0-quadrilateral.}
        Assume w.l.o.g. that this 0-quadrilateral lies at $e_0$. We consider the wedge-neighbors $t_2$ and $t_3$ of $f$ (see \cref{fig:1-triangle-case-2.1-2.2-b}). Observe that $\ch_3(t_2) = \ch_3(t_3) = -0.2$, and therefore $\ch_4(f) \ge 1 - 3 \cdot 0.3 - 2 \cdot 0.2 = -0.3$. 
        Note further that $f_0$ and $f_4$ cannot be 0-faces, and therefore, by \cref{pro:steps-g}, $f$ receives the missing charge in \textsf{Step 6} if $\vert f_0 \vert$ or $\vert f_4 \vert$ is at least five. The same holds if one of $f_0$ and $f_4$ is a 2-quadrilateral or both are 1-quadrilaterals, because in this case there is an excess of at least $0.3$ charge after \textsf{Step 5}, which is only contributed to $f$ (their other vertex-neighbor is $t_0$, which does not receive charge in \textsf{Step 5}, see \cref{fig:1-triangle-case-2.1-2.2-b}).
        
        Further $f_1$ and $f_2$ cannot be 2-triangles or 3-triangles. If both are 1-triangles, then there would be homotopic multi-edges, which is not allowed. So the last case to consider is when one of them -- w.l.o.g. $f_4$ -- is a 1-triangle and the other -- therefore $f_0$ -- is a 1-quadrilateral. If $f_4$ is the only wedge-neighbor, to which $f_0$ contributes in \textsf{Step 1-5}, then it contributes its excess of $0.3$ charge to $f$ in \textsf{Step 6} and $f$ is satisfied. Otherwise, the second wedge-neighbor of $f_0$ is also a 1-triangle and we have a planar cycle of length six (see \cref{fig:1-triangle-case-2.1-2.2-c}). Here, $f_0$ contributes $0.1$ charge to $f$ in \textsf{Step 6} and the 1-neighbor of $f_0$ that is a 2-triangle can contribute its excess of $0.2$ to $f$ in \textsf{Step 7}. Therefore, we have $\ch'(f) \ge 0$.

        \item \emph{Case 3: Exactly two 0-neighbors of $f$ are 0-quadrilaterals.} W.l.o.g. one 0-quadrilateral is at $e_0$. If the other 0-quadrilateral would be at $e_2$ (resp. $e_3$), then $e_1$ (resp. $e_4$) would have four crossings. Therefore, we can assume w.l.o.g. that the second 0-quadrilateral is at $e_4$. Here, we have $\ch_3(t_2) = -0.2$ as $f_1$ and $f_2$ are 2-triangles, thus $\ch_4(f) \ge 1- 4 \cdot 0.3 - 0.2 = -0.4$ (see \cref{fig:1-triangle-case-2.3-a}).
        
        We distinguish the type of the vertex-neighbor $f_4$. Note that $\vert f_4 \vert \ge 4$ and $f_4$ cannot be a 2-quadrilateral. If $f_4$ is not a 0-quadrilateral, 1-quadrilateral, 0-pentagon or 0-hexagon, then, by \cref{pro:steps-g}, $f_4$ contributes $0.4$ charge to $f$ in \textsf{Step 6}, and therefore $f$ is satisfied. The other cases are more complex, but they all have in common that if one of $f_0$ and $f_3$ is a 2-quadrilateral, then it has an excess of at least $2-2 \cdot 0.4 - 0.3 = 0.9$ charge after \textsf{Step 5} and this is enough to ensure $\ch_6(f) \ge 0$.       
    \begin{figure}
        \begin{subfigure}[b]{0.32\linewidth}
        \center
        \includegraphics[width=\textwidth, page=17]{3planarCases.pdf}
        \subcaption{}
        \label{fig:1-triangle-case-2.3-a}
        \end{subfigure}  
        \hfill
        \begin{subfigure}[b]{0.32\linewidth}
        \center
        \includegraphics[width=\textwidth, page=18]{3planarCases.pdf}
        \subcaption{}
        \label{fig:1-triangle-case-2.3-b}
        \end{subfigure} 
        \hfill
        \begin{subfigure}[b]{0.32\linewidth}
        \center
        \includegraphics[width=\textwidth, page=19]{3planarCases.pdf}
        \subcaption{}
        \label{fig:1-triangle-case-2.3-c}
        \end{subfigure}
        \hfill
        \caption{Illustrations for the proof of \cref{pro:steps-i} Case 3, 3.1 and 3.2.}
        \label{fig:1-triangle-case-2.3}
    \end{figure}
    \begin{itemize}
        \item \emph{Case 3.1: $f_4$ is a 0-quadrilateral.} Then the only case to consider is that $f_3$ and $f_4$ are 1-triangles. This directly implies a planar cycle of length seven (see \cref{fig:1-triangle-case-2.3-b}). Here, we make use of the second part of \textsf{Step 6} and have two 2-quadrilaterals contributing $0.9$ charge each to the 0-neighbors of $f$ at $e_0$ and $e_4$, which then is moved to $f$. Therefore, $f$ is satisfied after \textsf{Step 6}.
        
        \item \emph{Case 3.2: $f_4$ is a 1-quadrilateral.} Then $t_0$ and $t_4$ receive at least $0.2$ charge in \textsf{Step 3}, and therefore we have $\ch_4(f) \ge 1- 2 \cdot 0.3 - 3 \cdot 0.2 = -0.2$. If now $f_4$ contributes to less than two 1-triangles in \textsf{Step 4}, $f$ receives from $f_4$ enough charge in \textsf{Step 6}. Otherwise, $f_0$ and $f_3$ are 1-triangles, implying a planar cycle of length six (\cref{fig:1-triangle-case-2.3-c}). Here, $t_0$ and $t_4$ have two 1-neighbors that are 2-triangles and $\ch_3(t_0) = \ch_3(t_4) = -0.1$ holds. Therefore, $f$ contributes only $2 \cdot 0.1 + 0.2 + 2 \cdot 0.3 = 1$ charge and $f$ never has a demand.
        
        \item \emph{Case 3.3: $f_4$ is a 0-pentagon.}
        We introduce some new notation for $f_4$ and its wedge-neighbors, likewise for the 0-pentagon $f$ itself: Let $\tilde{f} := f_4$, $\tilde{e}_0$ the edge-segment of $\overline{e}_4$ at~$\tilde{f}$, $\tilde{e}_1$ the edge-segment of $\overline{e}_0$ at $\tilde{f}$ and so on (see \cref{fig:1-triangle-case-2.3.3-a}). Analogously, we denote by $\tilde{t}_i$ the wedge-neighbor of $\tilde{f}$ at $\tilde{e}_i$ and by $\tilde{f}_i$ the vertex-neighbor at the crossing of $\tilde{e}_i$ and $\tilde{e}_{(i + 1 \mod 5)}$. Note that $\tilde{t}_0 = f_0$ and $\tilde{t}_1 = f_3$ are 1-triangles or 2-quadrilaterals and, as pointed out above, we only have to consider the case that both are 1-triangles.
        
        Observe that $f$ is the only vertex-neighbor of $f_4$ that may receive charge from $f_4$ in \textsf{Step 6}, as all its other vertex-neighbors cannot be 0-faces.
        Distinguish the number of 1-triangles that are wedge-neighbors of $f_4$. Note that the wedge-neighbors of $f_4$ can never be 0-triangles or 0-pentagons, so, by \cref{pro:steps-f}, they contribute $0.3$ charge to $f_4$ if they are not 1-triangles. If three or less wedge-neighbors of $f_4$ are 1-triangles, then $\ch_5(f) \ge 1-3 \cdot 0.3 + 2 \cdot 0.3 \ge 0.7$, which then is contributed to $f$ in \textsf{Step 6} implying $\ch_6(f) \ge 0$. If all five wedge-neighbors of $f_4$ are 1-triangles, then we have the $F^3_6$ configuration, which is forbidden.
        So the case remains that four wedge-neighbors of $f_4$ are 1-triangles. Here, $\ch_5(f_4) \ge 1- 4 \cdot 0.3 + 0.3 = 0.1$ holds and this charge is contributed to $f$ in \textsf{Step 6}, so there is only $0.3$ charge missing for $f$.
                    
        By symmetry, $\tilde{t}_2$ is w.l.o.g. a 1-triangle.
        If $\tilde{t}_3$ is the wedge-neighbor of $\tilde{f}$ that is not a 1-triangle, then it must be 2-quadrilateral and this implies a planar cycle of length seven, in which $f$ is the only face with a demand after \textsf{Step 6} (see \cref{fig:1-triangle-case-2.3.3-b}). The 2-quadrilateral $\tilde{t}_3$ has an excess of $0.9$ charge after \textsf{Step 6} and contributes it in \textsf{Step 7} to $f$. Therefore, $f$ is satisfied.
        
        So assume now that $\tilde{t}_3$ is a 1-triangle and $\tilde{t}_4$ is the wedge-neighbor that is not a 1-triangle (see \cref{fig:1-triangle-case-2.3.3-c}).
        Note that $\tilde{t}_4$ is not a 0-face.
        So for all cases, except that $\tilde{t_4}$ is a 1-quadrilateral or 2-quadrilateral, \cref{pro:steps-g} guarantees that $\tilde{t_4}$ contributes in \textsf{Step 6} $0.4$ charge to all its vertex-neighbors. In particular, the 0-neighbor of $f$ at $e_0$ receives $0.4$ charge and gives it completely to $f$. Thus, in this case, $f$ is satisfied.
        
        If $\tilde{t_4}$ is a 2-quadrilateral, then we have $\ch_5(t_4) = 0.9$ and it contributes in the same way enough charge to $f$ via the 0-neighbor of $f$ at $e_0$. This works also if $\tilde{t_4}$ is a 1-quadrilateral contributing to only one wedge-neighbor (namely $\tilde{f}$) in \textsf{Step 1-5}.
        
        In the last case where $\tilde{t_4}$ is a 1-quadrilateral and contributes to $\tilde{f}$ and another wedge-neighbor in \textsf{Step 1-5}, this second wedge-neighbor is $\tilde{f}_3$ and must be a 1-triangle. This implies a planar cycle of length seven (see \cref{fig:1-triangle-case-2.3.3-d}). In this case, $\tilde{t_4}$ and its 1-neighbor that is a 2-triangle have an excess of $0.1$ resp. $0.2$ charge after \textsf{Step 5} and contribute it to $f$ in \textsf{Step 6} and \textsf{Step 7}. Therefore, $f$ is satisfied.
        \begin{figure}
        \begin{subfigure}[b]{0.24\linewidth}
        \center
        \includegraphics[width=\textwidth, page=20]{3planarCases.pdf}
        \subcaption{}
        \label{fig:1-triangle-case-2.3.3-a}
        \end{subfigure}  
        \hfill
        \begin{subfigure}[b]{0.24\linewidth}
        \center
        \includegraphics[width=\textwidth, page=21]{3planarCases.pdf}
        \subcaption{}
        \label{fig:1-triangle-case-2.3.3-b}
        \end{subfigure} 
        \hfill
        \begin{subfigure}[b]{0.24\linewidth}
        \center
        \includegraphics[width=\textwidth, page=22]{3planarCases.pdf}
        \subcaption{}
        \label{fig:1-triangle-case-2.3.3-c}
        \end{subfigure} 
        \hfill
        \begin{subfigure}[b]{0.24\linewidth}
        \center
        \includegraphics[width=\textwidth, page=23]{3planarCases.pdf}
        \subcaption{}
        \label{fig:1-triangle-case-2.3.3-d}
        \end{subfigure}
        \hfill
        \caption{Illustrations for Case 3.3 in the proof of \cref{pro:steps-i}.}
        \label{fig:1-triangle-case-2.3.3}
    \end{figure}
            
        \item \emph{Case 3.4: $f_4$ is a 0-hexagon.} Note that no wedge-neighbor of $f_4$ can be a 0-face, so $f_4$ contributes no charge in \textsf{Step 1-3} and \textsf{Step 5} (see \cref{fig:1-triangle-case-2.3.4-a}).  
        If at most four wedge-neighbors of $f_4$ are 1-triangles, then $\ch_5(f_4) \ge 2- 4 \cdot 0.3 = 0.8$ holds by \cref{pro:steps-f}. 
        In this case, there is at most one other vertex-neighbor of $f_4$ besides $f$ that can be a 0-quadrilateral or 0-pentagon and $f_4$ can contribute to both $0.4$ charge in \textsf{Step 6}. That is enough to satisfy $f$.

        If five wedge-neighbors of $f_4$ are 1-triangles, then no vertex-neighbor of $f_4$ except $f$ is a 0-face. Therefore, $f$ receives the excess of $f_4$ in \textsf{Step 6}, which is at least $2-5 \cdot 0.3 = 0.5$. So again $f$ is satisfied.

        Assume now that all six wedge-neighbors of $f_4$ are 1-triangles (see \cref{fig:1-triangle-case-2.3.4-b}). Then two of them have a demand of only $0.2$ after \textsf{Step 3} as they have two 1-neighbors that are 2-triangles. Therefore, $\ch_5(f_4) \ge 2 - 4 \cdot 0.3 - 2 \cdot 0.2 = 0.4$. Here, $f$ is the only face to which $f_4$ contributes in \textsf{Step 6} and we have $\ch_6(f) \ge 0$.
    \end{itemize}
    \begin{figure}
    \hfil
        \begin{subfigure}[b]{0.24\linewidth}
        \center
        \includegraphics[width=\textwidth, page=24]{3planarCases.pdf}
        \subcaption{}
        \label{fig:1-triangle-case-2.3.4-a}
        \end{subfigure}
        \hfil
        \begin{subfigure}[b]{0.24\linewidth}
        \center
        \includegraphics[width=\textwidth, page=25]{3planarCases.pdf}
        \subcaption{}
        \label{fig:1-triangle-case-2.3.4-b}
        \end{subfigure}
        \hfil
        \caption{Illustrations for Case 3.4 in the proof of \cref{pro:steps-i}.}
        \label{fig:1-triangle-case-2.3.4}
    \end{figure}
    \end{itemize}
\end{proof}

\section{Details for Section \ref{sec:linear-bounds}}\label{app:linear-bound}

\prooptface*

\begin{proof}
Insert the $m_0$ missing planar edges to $D_2$ at the boundaries of the forbidden configurations.
\begin{figure}
    \hfil
    \center
    \includegraphics[width=0.35\textwidth, page=2]{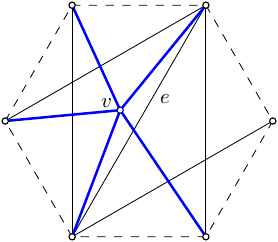} 
    \hfil
    \caption{Illustration for the proof of \cref{pro:optimal-faces}. We augment each $F^3_6$ configuration after the deletion of the three blue edges in \cref{fig:linear-bound} by one vertex and five edges for the drawing $\tilde{D}$.}
    \label{fig:linear-bound-2}
\end{figure}
Further add a vertex $v$ and five edges in every $F^3_6$ configuration from $D_3$ as shown in \cref{fig:linear-bound-2}. More precisely, notice that in $D_2$ three edges have been deleted from each $F^3_6$ configuration. Those three edges form a path consisting of a 2-hop edge, a 3-hop edge and a second 2-hop edge. Only one 3-hop edge $e$ still exists and it is crossing-free in $D_2$.
We arbitrarily choose a side of $e$ and place the new vertex $v$ close to $e$ at this side. We realize the five new edges by connecting $v$ to the two vertices of the configuration that are on same side of $e$, further to the two endpoints of $e$, and to one of the two endpoints on the opposite side of $e$.
We do not create new forbidden configurations by this operation, thus the number of $F^2_5$ and $F^2_6$ configurations in $D_2$ does not change.

As a next step, we remove one edge from each $F^2_5$ and $F^2_6$ configuration in $D_2$ and call this drawing $\tilde{D}$. Remark that $\tilde{D}$ is 2-planar, $F^2_5$-free, $F^2_6$-free and has 
$5(n-2) -2m_3 - m_{3-} + m_0 + 5m_3 - (c_{pent} + c_{hex})$
edges on $n + m_3 $ vertices.

Assume we have fewer $F^2_5$ and $F^2_6$ configurations in $D_2$ than stated in the proposition. Then $\tilde{D}$ would have more than
\begin{align*}
	5(n-2) -2m_3-m_{3-}+ m_0 + 5m_3 - \left[\frac{2}{3}(n-2) - \frac{4}{3}m_3 - m_{3-} +m_0\right]= \frac{13}{3} (n-2 + m_3)
\end{align*}
edges, which contradicts the statement of \cref{th:density-2-planar-forbidden-configuration}  for $\tilde{D}$.
\end{proof}

\promfour*
\begin{proof}
    To prove the result, we use the discharging technique. We assign in total $m_{3} +  c_{hex} + 4m_0 + 2c_\Delta$ charge to the network. In particular, 
    we  assign 1 charge to each $F^2_6$ and $F^3_6$ configuration, 4 charge to each edge in $E_0$ and 2 charge to every triangle in $\mathcal{H}'$. The proposition will follow from a redistribution of the charge in a way that each edge in $E_4$ has 1 charge and no edges, triangles or configurations have negative charge.
    We discuss the redistribution and the final charges by case analysis.
    For that, we use the fact that edges of $E_4$ do not cross each other and $D_4$ is 4-planar.
    \begin{figure}
    \hfil
    \begin{subfigure}[b]{0.35\linewidth}
    \center
    \hfil
    \includegraphics[width=\textwidth, page=3]{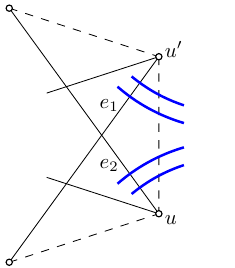}
    \subcaption{}
    \label{fig:linear-bound-charging-a}
    \end{subfigure}  
    \hfil
    \begin{subfigure}[b]{0.6\linewidth}
    \center
    \hfil
    \includegraphics[width=\textwidth, page=5]{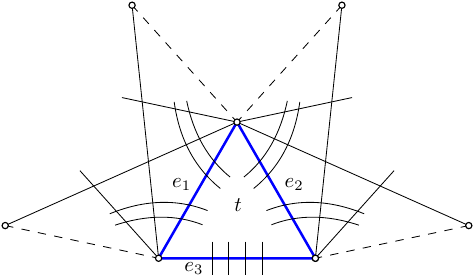}
    \subcaption{}
    \label{fig:linear-bound-charging-b}
    \end{subfigure}  
    \hfil
    \caption{(a) At most four edges of $E_4$ (blue) can leave a forbidden configuration $P$ through the same edge of its boundary, as otherwise one of the 2-hops $e_1,e_2$ of $P$ has~more than four crossings. (b) A triangle of $\mathcal{H}$ with all three edges in $E_4$ and two edges as boundary edges of forbidden configurations leads to a contradiction.}
    \label{fig:linear-bound-charging}
\end{figure}
\begin{itemize}
	\item[1.] \emph{$e \in E_4$ lies completely in one of the forbidden configurations.} This can only be the case in an $F^2_6$ or $F^3_6$ configuration as all five edges of an $F^2_5$ configuration still exist in $D_2$. 
    In each $F^2_6$ or $F^3_6$ configuration all 2-hops exist in $D_3$. 
    Therefore, $e$ is a 3-hop and crosses the other 3-hops inside the hexagon, which therefore cannot be in $E_4$.
    So $e$ is the only edge in $E_4$ inside the forbidden configuration and can receive 1 charge from it.
    \item[2.] \emph{$e \in E_4$ starts in a forbidden configuration $P$ and ends in another one, say $P'$.} Let $uu'$ be the edge on the boundary of $P$ that $e$ crosses. We will move 1 charge from $uu'$ to $e$ and argue that $uu' \in E_0$. Let $e_1,e_2$ and $e'_1,e'_2$ resp. be the 2-hop edges of $P$ and $P'$ that enclose the edge $uu'$ (see \cref{fig:linear-bound-charging-a}).
    Each of these four edges is crossed at least twice by edges belonging to the same forbidden configurations $P$ or $P'$.    
    Edge $e$ crosses at least two of those four edges. And since those edges must not be crossed more than four times, there are at most four edges of $E_4$
    that receive charge from the same boundary edge $uu'$. 
    We can also guarantee $uu' \in E_0$, as otherwise $e$ has at least five crossings (two each in the forbidden configurations and one with $uu'$).
	\item[3.] \emph{$e \in E_4$ is completely outside of any forbidden configuration.}
    By the properties of triangulation $\mathcal{H}$, $e$ is an edge of two neighboring triangles $t, t'\in \mathcal{H}$ and wlog $t\in \mathcal{H}'$. If $t'\notin \mathcal{H}'$, then we move 1 charge from $t$ to $e$. In the case $t'\in \mathcal{H}'$, $t$ and $t'$ both contribute $0.5$ charge to $e$ so that $e$ receives 1 charge in total. By this, the only critical case to consider where a triangle of $\mathcal{H}'$ may get negative charge is if all three edges $e_1,e_2,e_3$ of $t$ are in $E_4$ and two or three neighbors of $t$ are not in $\mathcal{H}'$.\\
    Assume wlog that the neighboring triangles of $t$ at $e_1$ and $e_2$ are not in $\mathcal{H}'$. Observe that in this case the eight edges crossing $e_1,e_2$ in $D_4$ must end at $t$ or leave $t$ through $e_3$ as otherwise they have more than four crossings (see \cref{fig:linear-bound-charging-b}). Thus at least four of these edges end at $t$ implying more than four crossings for the other edges. Therefore this case can not occur.

    \begin{figure}
    \hfil   
    \begin{subfigure}[b]{0.35\linewidth}
    \center
    \hfil
    \includegraphics[width=\textwidth, page=8]{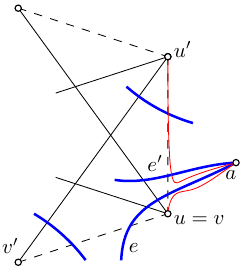}
    \subcaption{}
    \label{fig:linear-bound-charging-case-4-a}
    \end{subfigure}
    \hfil
    \begin{subfigure}[b]{0.35\linewidth}
    \center
    \hfil
    \includegraphics[width=\textwidth, page=6]{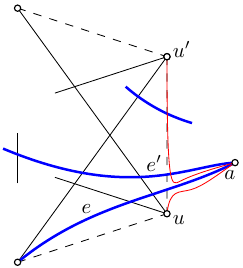}
    \subcaption{}
    \label{fig:linear-bound-charging-case-4-b}
    \end{subfigure}  
    \hfil    
    \caption{Illustrations for the subcase that $uu'$ is crossed three or four times by edges of $E_4$. (a) If $e$ crosses $uu'$ and $vv'$ then these edges are crossed by at most four distinct edges of $E_4$ in total. (b) $e'$ already has four crossings while $e$ has only three. Every additional edge crossing $e$ implies more crossings for $e'$.}
    \label{fig:linear-bound-charging-case-4}
    \end{figure}

	\item[4.] \emph{$e \in E_4$ lies partially in the faces of $\mathcal{H}'$ and a forbidden configuration.} This is the remaining case.
    Let $uu'$ be a boundary edge of the forbidden configuration $P$ that is crossed by edge $e \in E_4$. 
    If $uu'$ does not exist in $D_2$, then $e$ (and at most three other such edges) receive 1 charge each from $uu'$ (as in Case 2).\\
    Assume now, that $uu'$ exists in $D_2$. Recall that therefore there exists the charging triangle $t$ of $uu'$ with vertices $u,u',a$.
    Note that, by the choice of the triangulation $\mathcal{H}'$, Case~4 can only occur on one of the edges of the triangle $t$, here the edge $uu'$ (an edge in $E_4$ that enters $t$ through another edge and ends at $u$ or $u'$ would cross the pillar of $t$, contradiction). We distinguish three cases:
    \begin{itemize}
        \item \emph{$uu'$ is crossed by exactly one edge of $E_4$.} 
        If at most one edge of $t$ is in $E_4$, edge $e$ (and the other edge) can receive 1 charge from $t$. Assume now that the two edges $ua, u'a$ of $t$ are in $E_4$. As $e$ has at least two crossings at the forbidden configuration and one crossing with $uu'$, it has only one possible crossing inside $t$ left. Thus three edges crossing $ua$ and three edges crossing $u'a$ must leave triangle $t$ through $uu'$, a contradiction to the fact that $D_4$ is 4-planar.        
        \item \emph{$uu'$ is crossed by exactly two edges of $E_4$.} Let $e,e' \in E_4$ be the edges crossing $uu'$. If no edge of $t$ is in $E_4$, then $t$ has enough charge to contribute 1 charge to $e,e'$ each. Assume now the opposite, i.e., there is an edge $e_1 \in E_4$ of $t$. As in the case above, $e,e'$ can have only one crossing inside $t$ in $D_4$ and therefore three of the four edges crossing $e_1$ must leave $t$ through $uu'$. But this is a contradiction to the fact that $D_4$ is 4-planar.
        \item \emph{$uu'$ is crossed more often.} As before, we observe that there are at most four edges crossing $uu'$.
        Let $\tilde{e}_1$, $\tilde{e}_2$ be the edge-segments enclosing the 2-triangle next to $uu'$ in the forbidden configuration. Wlog $\tilde{e}_1$ is crossed by two edges $e,e' \in E_4$ and is incident to $u$. Let $a,a'$ be the endpoints of $e,e'$ that lie outside of the forbidden configuration.
        \begin{itemize}
            \item \emph{$a = a'$.} Assume first that $e$ also leaves the forbidden configuration through another boundary edge $vv'$. We have $u=v$, as otherwise $e$ is crossed more than four times. 
            Every edge of $E_4$ crossing $uu'$ or $vv'$ must cross one of the edges $u'v'$ or $\tilde{e}_1$. By 4-planarity, both these edges can have two crossings each additional to the crossings with other edges of the forbidden configuration. Thus, there are at most four edges in $E_4$ that cross $uu'$ or $vv'$ (see \cref{fig:linear-bound-charging-case-4-a}). The charging triangle $t$ can contribute 2 charge to these edges, so 2 charge more are required. We find that at $vv'$, if it does not exist in $D_2$, or otherwise at the charging triangle of $vv'$.
            
            Assume now that $e$ does not cross another boundary edge. Then it is impossible that $e$ and $e'$ both have exactly four crossings in $D_4$ while the drawing is still 4-planar (see \cref{fig:linear-bound-charging-case-4-b}).
            
            \item \emph{$a \ne a'$.} Recall that in this case there exists also the charging by-triangle $t'$ of $uu'$. No edge of $t,t'$ is in $E_4$, as again this would imply that $e,e'$ or $uu'$ has more than four crossings. Since $t,t'$ uniquely belong to $uu'$, they can contribute 1 charge each to $e$ and the at most three other edges in $E_4$ crossing $uu'$.
        \end{itemize}      
    \end{itemize}
  
\end{itemize}
\end{proof}
\end{document}